\documentclass[10pt,reqno]{amsart}

\usepackage{amsmath,amscd,amssymb }

\usepackage[backend=bibtex,style=numeric-comp,maxbibnames=10]{biblatex}
\usepackage{enumerate}
\usepackage{comment}
\bibliography{twisted.bib}

\newtheorem{proposition}{Proposition}[section]
\newtheorem{lemma}[proposition]{Lemma}
\newtheorem{theorem}[proposition]{Theorem}
\newtheorem{corollary}[proposition]{Corollary}
\newtheorem{conjecture}{Conjecture}[section]

\theoremstyle{definition}
\newtheorem{remark}[proposition]{Remark}
\newtheorem{definition}[proposition]{Definition}

\newcommand{\R}{\mathbb{R}}
\newcommand{\C}{\mathbb{C}}

\newcommand{\Z}{\mathbb{Z}}
\newcommand{\Q}{\mathbb{Q}}

\newcommand{\pr}{\mathbb{P}}

\newcommand{\scK}{\mathcal{K}}
\newcommand{\scL}{\mathcal{L}}

\newcommand{\scI}{\mathcal{I}}
\newcommand{\scD}{\mathcal{D}}
\newcommand{\scH}{\mathcal{H}}
\newcommand{\scM}{\mathcal{M}}
\newcommand{\scT}{\mathcal{T}}
\newcommand{\scX}{\mathcal{X}}
\newcommand{\scO}{\mathcal{O}}
\newcommand{\scR}{\mathcal{R}}
\newcommand{\scB}{\mathcal{B}}
\newcommand{\scY}{\mathcal{Y}}
\newcommand{\scZ}{\mathcal{Z}}
\newcommand{\scP}{\mathcal{P}}

\DeclareMathOperator{\lct}{lct}

\DeclareMathOperator{\Ric}{Ric}

\DeclareMathOperator{\DF}{DF}
\DeclareMathOperator{\Amp}{Amp}

\DeclareMathOperator{\wt}{wt}
\DeclareMathOperator{\rk}{rk}
\DeclareMathOperator{\tr}{tr}

\title[Uniform stability of twisted cscK metrics]{Uniform stability of twisted constant scalar curvature K\"ahler metrics}
\author{Ruadha\'i Dervan}

\begin{document}

\begin{abstract} 

We introduce a norm on the space of test configurations, which we call the minimum norm. We conjecture that uniform K-stability with respect to this norm is equivalent to the existence of a constant scalar curvature K\"ahler metric. This notion of uniform K-stability is analogous to coercivity of the Mabuchi functional. We characterise the triviality of test configurations, by showing that a test configuration has zero minimum norm if and only if it has zero $L^2$-norm, if and only if it is almost trivial. 

We prove that the existence of a twisted constant scalar curvature K\"ahler metric implies uniform twisted K-stability with respect to the minimum norm, when the twisting is ample.

We give algebro-geometric proofs of uniform K-stability in the general type and Calabi-Yau cases, as well as in the Fano case under an alpha invariant condition. Our results hold for line bundles sufficiently close to the (anti)-canonical line bundle, and also in the twisted setting. We show that log K-stability implies twisted K-stability, and also that twisted K-semistability of a variety implies that the variety has mild singularities. 
\end{abstract}

\maketitle

\tableofcontents

\section{Introduction}

One of the central problems in complex geometry is to relate the existence of certain canonical metrics with algebro-geometric notions of stability. For vector bundles, the Hitchin-Kobayashi correspondence states that the existence of a Hermite-Einstein metric is equivalent to slope-polystability. The analogous problem for manifolds is a conjecture due to Yau-Tian-Donaldson. The metric in this case is a constant scalar curvature K\"ahler (cscK) metric on a polarised manifold $(X,L)$, and the form of stability is expected to be some modification of K-stability. This conjecture is now proven in in the Fano case, so that the metric is K\"ahler-Einstein \cite{RB,CDS,GT1}, and it is known in general that the existence of a cscK metric implies K-stability when the manifold has discrete automorphism group \cite{D2, JS2}. However, there are examples suggesting the converse is not true in general \cite{ACGTF}, and that one needs to strengthen the definition of K-stability to imply the existence of a cscK metric \cite{GS3}.

Roughly speaking, to define K-stability one considers embeddings of $X$ into projective space via global sections of $L^k$ and flat degenerations under these embeddings; these are called test configurations. Each test configuration has an associated weight, called the Donaldson-Futaki invariant, and K-stability requires that this weight is positive. Our modification is to define a norm on the space of test configurations, which we call the minimum norm. We define \emph{uniform K-stability with respect to the minimum norm} to mean that the Donaldson-Futaki invariant of a test configuration is bounded below by (a constant times) the minimum norm. The following is then a small refinement of the Yau-Tian-Donaldson conjecture.

\begin{conjecture}\label{introconjecture} A polarised manifold $(X,L)$ with discrete automorphism group admits a cscK metric if and only if it is uniformly K-stable with respect to the minimum norm. \end{conjecture}

Our motivation for this conjecture is twofold: firstly, we can prove this uniform K-stability holds in several geometric situations. Secondly, there is a strong link between uniform K-stability in this sense and coercivity of the Mabuchi functional, which we describe in Section \ref{introminnorm}.

There are three related themes in the present work, which we briefly describe before discussing each in detail.

The first theme is to study a form of stability related to \emph{twisted cscK metrics}. Here one takes a polarised manifold $(X,L)$ together with an auxiliary line bundle $T$, which we take to be semi-positive. We are then interested in solutions of the \emph{twisted cscK equation} \begin{equation}\label{introtwistedcscK}S(\omega) - \Lambda_{\omega} \alpha = C_{\alpha},\end{equation} with $\omega \in c_1(L)$ positive and $\frac{1}{2}\alpha\in c_1(T)$ is semi-positive. This equation arises in certain constructions of cscK metrics, most notably along the Aubin continuity path for constructing K\"ahler-Einstein metrics and Fine's construction of cscK metrics on fibrations \cite{F1,F2}. We define an analogue of stability in this case, and prove the following.

\begin{theorem}\label{intromaintheorem} Suppose $(X,L,T)$ admits a twisted cscK metric. \begin{itemize} \item[(i)] If $\alpha$ is positive, so that $T$ is ample, then $(X,L,T)$ is uniformly twisted K-stable with respect to the minimum norm. \item[(ii)] If $\alpha$ is semi-positive, then $(X,L,T)$ is twisted  K-semistable.\end{itemize} \end{theorem}

This answers questions of Donaldson \cite[Section 5, Remark 2]{D1} and Sz\'ekelyhidi \cite[Section 4]{GS2}. 

Our second theme is to prove that uniform K-stability holds in several geometric situations, motivating Conjecture \ref{introconjecture}.

\begin{theorem}\label{introuniformity} Let $X$ be a Kawamata log terminal variety, and suppose one of the following conditions holds.
\begin{itemize} 
\item[(i)] $X$ is Fano with $L=-K_X$ and Tian's alpha invariant satisfies $\alpha(X,-K_X)>\frac{n}{n+1}$,
\item[(ii)] $X$ is of general type with $L=K_X$,
\item[(iii)] $K_X\equiv 0$, i.e. $X$ is numerically Calabi-Yau, and $L$ is arbitrary.
\end{itemize} Then $(X,L)$ is uniformly K-stable with respect to the minimum norm.
 \end{theorem}
 
This strengthens results of Odaka \cite{O2} and Odaka-Sano \cite{OS}, who proved K-stability under the same hypotheses. We remark that it is known that the above varieties admit K\"ahler-Einstein metrics. However we also prove uniform K-stability in \emph{explicit neighbourhoods} of the (anti)-canonical classes on varieties of general type and also Fano varieties under an alpha invariant condition (building on \cite{RD}), for which the existence of a cscK metric is \emph{not} yet known (see Theorems \ref{introgt} and \ref{introalpha}). We also give direct algebro-geometric proofs of uniform \emph{twisted} K-stability in analogous situations.

The third theme of the present work is to characterise the \emph{triviality} of test configurations. It was noticed by Li-Xu that \emph{every} polarised variety $(X,L)$ admits test configurations with zero Donaldson-Futaki invariant \cite{LX}, which are not the trivial test configuration. Stoppa showed that these test configurations are characterised as having normalisation equivariantly isomorphic to the trivial test configuration \cite{JS3}. There have since been two competing definitions of K-stability: either one restricts to test configurations with normal total space or one requires that the $L^2$-norm of the test configuration is positive. We show that these two definitions are actually equivalent, building on work of Lejmi-Sz\'ekelyhidi \cite{LS}.

\begin{theorem}\label{intronorms} Let $(\scX,\scL)$ be a test configuration. The following are equivalent. \begin{itemize}\item[(i)] $(\scX,\scL)$ has normalisation equivariantly isomorphic to $(X\times\C, L)$ with the trivial action on $X$, i.e. $(\scX,\scL)$ is almost trivial. \item[(ii)] The $L^2$-norm $\|\scX\|_2$ is zero. \item[(iii)] The minimum norm $\|\scX\|_m$ is zero. \end{itemize}\end{theorem}

It follows that uniform K-stability indeed implies K-stability.

\subsection{The minimum norm and the Mabuchi functional.}\label{introminnorm} To motivate Conjecture \ref{introconjecture} further, we now discuss the link between uniform K-stability with respect to the minimum norm and coercivity of the Mabuchi functional, referring to \cite{XC} for an introduction to the Mabuchi functional. Let $(X,L)$ be a smooth polarised variety and fix some $\omega \in c_1(L)$. Denote by $$\scH(\omega) = \{\phi\in C^{\infty}(X,\R): \omega_{\phi}=\omega+i\partial\bar{\partial}\phi > 0\}$$ the space of K\"ahler potentials in $c_1(L)$. 

\begin{definition} The Mabuchi functional is defined as $$\scM_{\omega}(\phi)=-\int_0^1 \int_X \dot{\phi}_t(s(\omega_t) - n\mu(X,L))\omega_t^n\wedge dt,$$ where $\phi_t$ is any path in $\scH(\omega)$ joining $\omega$ to $\omega_{\phi}$. Here $$\mu(X,L)=\frac{-K_X.L^{n-1}}{L^n}=\frac{\int_Xc_1(X).c_1(L)^{n-1}}{\int_Xc_1(L)^n}$$ is the \emph{slope} of $(X,L)$. Defining and auxiliary functional \begin{align*}I_{\omega}(\phi) &=\int_X \phi (\omega^n - \omega_{\phi}^n),\end{align*} we say that the Mabuchi functional is \emph{coercive} if $$\scM_{\omega}(\phi) \geq a I_{\omega}(\phi) + b,$$ for constants $a,b\in \R$ with $a>0$. Coercivity in particular implies that the Mabuchi functional is bounded below. \end{definition}

The key feature of the Mabuchi functional is that, after defining a Riemannian metric on $\scH(\omega)$, one can show that the Mabuchi functional is convex along geodesics. Moreover, when they exist, its critical points are precisely the K\"ahler metrics of constant scalar curvature. This inspires the following conjecture.

\begin{conjecture}[Mabuchi, Tian]Suppose $X$ has discrete automorphism group. Then there exists a cscK metric in $c_1(L)$ if and only if the Mabuchi functional is coercive. \end{conjecture}

In \cite{XC}, Chen introduced the J-flow, which is related to the boundedness of another functional, which we denote $J(\omega, \eta)$. Here $\eta\in  c_1(T)$ is an arbitrary K\"ahler form. Chen showed that the convergence of the J-flow implies that $J(\omega, \eta)$ is bounded below \cite[Proposition 3]{XC}. On the other hand, \cite[Conjecture 1]{LS} states that a form of algebro-geometric stability should be equivalent to the convergence of the J-flow. This form of algebro-geometric stability assigns a weight $J_T(\scX,\scL)$ to each test configuration $(\scX,\scL)$, and the definition of the minimum norm is \emph{precisely} $J_L(\scX,\scL)$, setting $T=L$. We therefore expect that uniform K-stability with respect to the minimum norm should correspond to the existence of a lower bound of the Mabuchi functional in terms of $J(\omega, \omega)$. Setting $\eta=\omega$ in Chen's functional, one sees that this lower bound is equivalent to the coercivity of the Mabuchi functional. This leads us to make the following more precise version of Conjecture \ref{introconjecture}, which is a refinement of the Yau-Tian-Donaldson conjecture. 

\begin{conjecture} Let $(X,L)$ be a polarised manifold with discrete automorphism group. The following are equivalent. 
\begin{itemize} \item[(i)] There exists a cscK metric in $c_1(L)$. 
\item[(ii)] The Mabuchi functional is coercive in the K\"ahler class $c_1(L)$. 
\item[(iii)] $(X,L)$ is uniformly K-stable with respect to the minimum norm.
\end{itemize} \end{conjecture}

\subsection{Remarks on twisted cscK metrics and stability}

One of the novel features of Theorem \ref{intromaintheorem}, namely that the existence of a twisted cscK metric implies twisted K-stability, is that we are able prove strict twisted K-stability, even when $X$ admits automorphisms. This is not true the untwisted case, where one must adapt the definition of stability to allow the Donaldson-Futaki invariant to be zero when the test configuration arises from an automorphism. Our proof of strict stability is based on an elementary perturbation method, where we perturb the \emph{equation in question}, borrowing an idea from Lejmi-Sz\'ekelyhidi's study of the J-flow \cite{LS}. This is in contrast Stoppa's perturbation argument in the untwisted setting \cite{JS2}, which perturbs the \emph{manifold itself}, even at the level of topological spaces. We therefore avoid Arezzo-Pacard's deep results on blow-ups of manifolds admitting cscK metrics \cite{AP}. 

The closest previously known result to Theorem \ref{intromaintheorem} is due to Stoppa \cite{JS}, who showed that the existence of a twisted cscK metric implies twisted K-semistability with respect to test configurations obtained by performing deformation to the normal cone with respect to smooth divisors. Our result sharpens this by allowing arbitrary test configurations, and proving strict (in fact uniform) stability. 

The method of proof of Theorem \ref{intromaintheorem} initially follows Donaldson's beautiful lower bound on the Calabi functional \cite{D2} to prove semistability. As in his case, the key analytic result is the asymptotic expansion of a certain Bergman kernel, which provides the link between the differential and algebraic geometry. A difference compared to Donaldson's case is that while the Bergman kernel we use embeds $X$ into projective space using global sections of $L^k\otimes T^{-1}$, the twisted Donaldson-Futaki invariant is calculated using an embedding through global sections of $L^k$. We prove a twisted version of equivariant Riemann-Roch to relate the quantities in the two embeddings; this is dealt with in Lemma \ref{dgtoag}.  Our definition of a twisted test configuration then requires equivariant embeddings into projective space using global sections of \emph{both} line bundles $L^k$ and $L^k\otimes T^{-1}$. Nonetheless, we show in Proposition \ref{approximation} that to check twisted K-\emph{semi}stability, one can assume the test configuration embeds only through global sections of $L^k$. As outlined above, a perturbation argument then gives uniform stability, completing the proof.

We now state a twisted version of the Yau-Tian-Donaldson conjecture.

\begin{conjecture} A smooth polarised manifold $(X,L)$ together with an ample line bundle $T$ admits a twisted cscK metric if and only if it is uniformly twisted K-stable with respect to the minimum norm. \end{conjecture} 

\subsection{Applications and examples of twisted cscK metrics}

The twisted cscK equation $$S(\omega) - \Lambda_{\omega} \alpha = C_{\alpha}$$ appears in many constructions of cscK metrics. The original, and perhaps most important, use of twisted cscK case is in the construction of K\"ahler-Einstein metrics. Here one takes $L$ to be the anti-canonical bundle of $X$ and looks for solutions to the \emph{twisted K\"ahler-Einstein equation} \begin{equation}\label{twistedKE}\Ric \omega = \beta \omega  + (1-\beta)\omega_0,\end{equation} with $\omega,\omega_0\in c_1(X)$. Of course when $\beta = 1$ this is the usual K\"ahler-Einstein equation. 

In this setting, Sz\'ekelyhidi \cite{GS2} defined the \emph{greatest lower bound on the Ricci curvature} \begin{equation}\label{fanoKE}R(X) = \{\sup \beta: \mathrm{ \ equation \ }(\ref{twistedKE}) \mathrm{ \ is \ solvable }\}.\end{equation} This invariant can be thought of as measuring how far a Fano manifold is from admitting a K\"ahler-Einstein metric. Our results provide an algebro-geometric analogue of Sz\'ekelyhidi's invariant, by defining \begin{equation}\label{twistedKEkstab}S(X) = \left\{\sup \beta: \left(X,-K_X,-\frac{1-\beta}{2}K_X\right) \textrm{ \ is \ uniformly \ twisted \ K-stable}\right\}.\end{equation} By Remark \ref{realtwisting}, Theorem \ref{intromaintheorem} holds also when $T=c L$ for $c\in \R_{>0}$ is an $\R$-line bundle. As such we have the following.

\begin{corollary} Sz\'ekelyhidi's greatest lower bound on the Ricci curvature satisfies $$R(X)\leq S(X).$$\end{corollary}

The twisted Yau-Tian-Donaldson conjecture in this case states that $R(X)=S(X)$. By a recent result of Li, this conjecture is true in the case $\beta=1$ \cite{L1}. Sz\'ekelyhidi has also given upper bounds for $R(X)$ in explicit cases using twisted K-stability \cite{GS2}. We give a \emph{lower} bound for $S(X)$ using Tian's alpha invariant $\alpha(X,-K_X)$ in Corollary \ref{aubinkstab} as follows.

\begin{theorem}We have $S(X) \geq \min\{\alpha(X,-K_X)\frac{n+1}{n}, 1\}.$\end{theorem}

Another natural situation in which the twisted K\"ahler-Einstein equation appears is on varieties with semi-ample canonical class, an important class of varieties in birational geometry. When $X$ has ample canonical class, the Aubin-Yau theorem implies that $X$ admits a K\"ahler-Einstein metric. A natural question to ask is if there is an analogue of this result when the variety has semi-ample canonical class. One then has a morphism through global sections of the canonical class, the fibres of which are Calabi-Yau. Song-Tian \cite{ST1,ST2} answered this question affirmatively by showing that each variety with semi-ample canonical class admits a twisted K\"ahler-Einstein metric of the form \begin{equation}\label{twistedgentype}\Ric \omega-\alpha=-\omega,\end{equation}  where roughly $\alpha$ comes from a Weil-Petersson type metric on the moduli space of Calabi-Yau manifolds. In this case $\alpha$ is only semi-positive, so Theorem \ref{intromaintheorem} implies twisted K-semistability.

Twisted cscK metrics are also useful in constructing genuine cscK metrics. Consider a fibration $\pi: X\to B$ with a relatively ample line bundle $L_F$, such that all fibres $X_b$ admit a cscK metric on $L_F|_b$. Take an ample line bundle $L_B$ and assume moreover that $B$ admits a twisted cscK metric with $T = \pi_*(K_{X/ B} + c L_F)$, where $c$ is a topological constant. The main result of Fine \cite{F1, F2} is that in this case $X$ admits a cscK metric on $r\pi^*L_B + L_F$ in the adiabatic limit $r \gg 0$. Furthermore, in this setting Stoppa \cite{JS} has also used twisted K-stability as an obstruction for certain $B$, showing that fibrations over certain $B$ cannot admit cscK metrics in the adiabatic limit. 

\subsection{Sufficient geometric conditions for uniform and twisted K-stability} We provide several situations in which we can directly prove uniform and twisted K-stability. These conditions are primarily motivated by known results on the existence of (twisted) K\"ahler-Einstein metrics. A noteworthy feature is that we are able to prove uniform K-stability in explicit neighbourhoods of the (anti)-canonical class, in contrast to the analytic situation where these polarised varieties are not yet known to admit cscK metrics. The novelty of these results is the following: firstly, we prove \emph{uniform} K-stability in several situations in which only K-stability was previously known. Secondly, we also prove \emph{twisted} K-stability in similar situations.

Our first results are in the (twisted) general type case. In his original paper proving the Calabi conjecture, Yau \cite[Theorem 4]{STY2} showed that when $K_X+2T$ is ample there exists a twisted K\"ahler-Einstein metric solving equation (\ref{twistedgentype}) for \emph{any} positive $\frac{1}{2}\alpha \in c_1(T)$. We give algebro-geometric proofs of uniform twisted K-stability in similar cases. In order to ease notation, define the \emph{twisted slope} of $(X,L,T)$ to be $$\mu(X,L,T)=\frac{(-K_X-2T).L^{n-1}}{L^n}=\frac{\int_X(c_1(X)-2c_1(T)).c_1(L)^{n-1}}{\int_Xc_1(L)^n}.$$ The twisted slope is therefore a topological quanitity, the sign of which is governed by the geometry of $K_X+2T$. When $K_X+2T$ is ample, the slope is negative and we call this the twisted general type case. The twisted Calabi-Yau case is when $K_X+2T$ is numerically trivial and so the twisted slope is zero, and finally the twisted slope is positive in the twisted Fano case, i.e. when $-K_X-2T$ is ample. In particular, if $L=\pm(K_X+ 2T)$, then $\mu(X,L,T)=\mp 1$.

\begin{theorem}\label{introgt}Let $X$ be a $\Q$-Gorenstein log canonical variety. Suppose that \begin{equation*}-\mu(X,L,T) L\geq K_X+2T,\end{equation*} in the sense that the difference is nef. Then $(X,L,T)$ is uniformly twisted K-stable with respect to the minimum norm. \end{theorem} When $L=K_X+2T$ the slope condition is automatic and this is the algebro-geometric analogue of Yau's result. This result is new for general $L$ even when $T=\scO_X$. In the case $T=\scO_X$ and $L=K_X$, the uniformity that we prove strengthens work of Odaka \cite{O2}. In the twisted Calabi-Yau case we have the following.

\begin{theorem}\label{introcy}Let $X$ be a $\Q$-Gorenstein variety with canonical divisor $K_X$. Suppose that $K_X+2T$ is numerically trivial, and let $L$ be an arbitrary ample line bundle.  \begin{itemize} \item[(i)] If $X$ is log canonical, then $X$ is twisted K-semistable. \item[(ii)] If $X$ is Kawamata log terminal, then $X$ is uniformly twisted K-stable with respect to the uniform norm.\end{itemize} \end{theorem}

When $T=\scO_X$ this again strengthens work of Odaka \cite{O2} to uniform K-stability. We prove in Theorem \ref{alpha} a corresponding result when $-K_X-2T$ is ample using an alpha invariant type condition, giving as a special case an algebro-geometric analogue of a result of Berman \cite[Theorem 4.5]{RB2}. 

\begin{theorem}\label{introalpha} Let $(X,L,T)$ be a $\Q$-Gorenstein Kawamata log terminal variety $X$ with canonical divisor $K_X$. Suppose that 
\begin{itemize}
\item[(i)]  $\alpha(X,L)>\frac{n}{n+1}\mu(X,L,T)$ and
\item[(ii)]  $-(K_X +2T)\geq \frac{n}{n+1}\mu(X,L,T) L$.
\end{itemize}
Then $(X,L,T)$ is uniformly twisted K-stable with respect to the minimum norm.\end{theorem}

Again when $T=\scO_X$, the uniformity that we prove strengthens work of Odaka-Sano \cite{OS} when $L=-K_X$ and the author \cite{RD} for general $L$.

\subsection{K-stability and K\"ahler-Einstein metrics} There has been much recent work on the study of K\"ahler-Einstein metrics with cone angles along a divisor \cite{D3}. Indeed, in the solution of the Yau-Tian-Donaldson conjecture when $L=-K_X$, an important feature was the use of the \emph{Donaldson continuity method} \begin{equation}\label{doncont}\Ric \omega = \beta \omega + (1-\beta)\{D\},\end{equation} where $\{D\}$ is the current of integration along $D$. The notion of stability in this case is called log K-stability, and a result of Berman \cite[Theorem 4.2]{RB2} states that the existence of a solution to equation (\ref{doncont}) implies log K-stability. This was used in a fundamental way in the proof of the Yau-Tian-Donaldson conjecture in this case. 

It is tempting to ask if one can show that K-stability implies the existence of a K\"ahler-Einstein metric using instead the Aubin continuity method, i.e. equation (\ref{twistedKE}), avoiding the use of metrics with cone singularities. One of the most important steps of such a proof, a partial $C^0$-estimate along the continuity method, has recently been proven by Sz\'ekelyhidi \cite{GS4}. Theorem \ref{intromaintheorem} provides another result which would be required in adapting the methods of \cite{CDS,GT1} to the Aubin continuity method.

\subsection{Twisted and log K-stability} The twisted Donaldson-Futaki invariant of a fixed test configuration with respect to some twisting $T$ is precisely the log Donaldson-Futaki invariant for a general $D \in |2T|$. That the existence of a twisted cscK metric depends on an arbitrary choice of twisting $\alpha$ is analogous to the arbitrary choice of $D$ in the study of cscK metrics with cone singularities along a divisor. We relate the two notions of stability in Theorem \ref{logimpliestwisted}.

\begin{theorem} Suppose $(X,L,D)$ is log K-stable, with $D\in |2T|$. Then $(X,L,T)$ is twisted K-stable. \end{theorem}

A similar result was proven by Sz\'ekelyhidi \cite[Theorem 6]{GS} in the case $L=-K_X$, where it is also shown that the log K-stability is not equivalent to twisted K-stability. Our proof also gives an explicit expression for the difference in the respective Donaldson-Futaki invariants. 

\subsection{Singularities and moduli} An important feature of twisted cscK metrics is that they are unique, when they exist \cite{BB}. As such one might hope to use such metrics to form twisted moduli spaces. In the Fano setting, recent work of Odaka \cite{O1} and Odaka-Spotti-Sun \cite{OSS} has developed the idea that one can form, and even compactify, moduli spaces of  K\"ahler-Einstein manifolds. While not all Fano manifolds admit K\"ahler-Einstein metrics, every Fano manifold admits a twisted K\"ahler-Einstein metric with some parameter $\beta$, as in equation (\ref{twistedKE}). It would be interesting to see if one could extend the existence of moduli of K\"ahler-Einstein manifolds to the twisted case with some fixed parameter $\beta$. This would allow a more general class of manifolds in the moduli space. With applications to the compactification of moduli spaces in mind, we relate twisted K-stability to the singularities of $X$ in Section \ref{singularitiessection} as follows.

\begin{theorem}Let $X$ be a normal variety together with line bundles $L,T$. \begin{itemize} \item[(i)] Suppose $(X,L,T)$ is twisted K-semistable. Then $X$ has log canonical singularities. \item[(ii)] Suppose $\mu(X,L,T) L + K_X + 2T$ is nef. Then $X$ has Kawamata log terminal singularities. \end{itemize}\end{theorem}

This result is due to Odaka \cite{O4} when $L=-K_X$ and $T=\scO_X$, so that the slope condition is automatically satisfied. When $T=\scO_X$, part $(ii)$ of the previous Theorem is new for general $L$. This leads to the following converse of Theorem \ref{introgt} and Theorem \ref{introcy}.

\begin{corollary} Let $X$ be a normal variety together with line bundles $L,T$. \begin{itemize}\item[(i)] If the twisted slope satisfies $$-\mu(X,L,T) L\geq K_X+2T,$$ then $X$ is uniformly twisted K-stable with respect to the minimum norm if and only if $X$ is log canonical.\item[(ii)] If the twisted slope is zero, i.e. $(X,L,T)$ is numerically twisted Calabi-Yau, then $X$ is twisted K-semistable if and only if $X$ is log canonical. \end{itemize}\end{corollary}

\subsection{Relation to the work of Boucksom-Hisamoto-Jonsson} After completing the present work, the author learned that some of the uniformity results have been independently proven by Boucksom-Hisamoto-Jonsson \cite{BHJ}. In particular, they prove Theorem \ref{intronorms} about the triviality of norms when the test configuration is normal. They also prove uniform K-stability holds on Calabi-Yaus, varieties of general type and Fanos under an alpha invariant condition, analogously to Theorem \ref{introuniformity}. They describe in more detail the link between the minimum norm and coercivity of the Mabuchi functional, and show that the minimum norm is \emph{not} Lipshitz equivalent to the $L^2$-norm. They also show a polarised variety $(X,L)$ can \emph{never} be uniformly K-stable with respect to the $L^p$-norm unless $p\leq \frac{n}{n-1}$; this gives more evidence that the minimum norm is the correct norm with which to define uniform K-stability. The author thanks S\'ebastien Boucksom for sending him a draft of their work.

\ \\
\noindent {\bf Notation and conventions:} We often use the same letter to denote a divisor and the associated line bundle, and mix multiplicative and additive notation for line bundles. We omit certain factors of $2\pi$ throughout for notational convenience. A line bundle $T$ is called semi-positive if it admits a smooth curvature $(1,1)$-form $\alpha\in c_1(T)$ which is positive semi-definite. In particular, semi-ample $\Rightarrow$ semi-positive $\Rightarrow$ nef.

\ \\
\noindent {\bf Acknowledgements:} I would like to thank my supervisor Julius Ross for many useful discussions. I would like to thank Yoshinori Hashimoto and Kento Fujita for helpful comments and also Mehdi Lejmi and Gabor Sz\'ekelyhidi for answering several questions related to \cite{LS}, which contains many ideas which were invaluable to the present work.

The author was funded by a studentship associated to an EPSRC Career Acceleration Fellowship (EP/J002062/1).

\section{K-stability of twisted cscK metrics}
\subsection{Notions of stability} 

Let $X$ be a normal projective variety of dimension $n$, together with two line bundles $L,T$. We assume throughout that $L$ is ample.

\begin{definition}A twisted test configuration for $(X,L,T)$ is a triple $(\scX,\scL,\scT)$, where 
\begin{itemize}
\item $\scX$ is a scheme together with a proper flat morphism $\pi: \scX \to \C$,
\item there is a $\C^*$-action on $\scX$ covering the natural action on $\C$,
\item $\scL$ and $\scT$ are equivariant line bundles with respect to the $\C^*$-action with $\scL$ relatively ample,\end{itemize}
such that each fibre of $\pi$ over $t\neq 0$ is isomorphic to $(X,L^r,T^s)$ for some $r,s>0$. \end{definition}

\begin{remark} We assume for notational simplicity throughout that $r=s=1$, our results will be invariant under scaling so this will not cause issue. \end{remark}

Since the $\C^*$-action on a twisted test configuration fixes the central fibre $(\scX_0,\scL_0)$, there is an induced $\C^*$-action on $H^0(\scX_0, \scL_0^k)$ for $k \gg 0$ with infinitesimal generator $A_k$. By general theory, the dimension $\dim H^0(\scX_0, \scL_0^k)$ and the total weight of the $\C^*$-action on $H^0(\scX_0,\scL_0^k)$ are polynomials for $k\gg 0$. Denote these polynomials respectively by \begin{align*}
h(k)&=a_0k^n+a_1k^{n-1}+O(k^{n-2}), \\
w(k) &=\tr(A_k)= b_0k^{n+1}+b_1k^n+O(k^{n-1}).\end{align*}

Suppose for the moment that $T$ is very ample, and let $D \in |T|$ be an arbitrary divisor. Denote by $\scD$ the closure of $D$ under the $\C^*$-action on $\scX$. By \cite[Proposition 9.7]{RH}, since taking the closure adds no embedded points \cite[Tag 083P]{stacks-project}, $\scD\to\C$ is flat with respect to any polarisation.  As such, $\scD\to\C$ defines another test configuration for $D$ and there are corresponding Hilbert and weight polynomials \begin{align*}
&\hat{h}(k)=\dim H^0(D,L|_D^k) = \hat{a}_0k^{n-1}+O(k^{n-2}), \\
&\hat{w}(k) =\wt H^0(\scD_0,\scL_0|_{\scD_0}^k)= \hat{b}_0k^{n}+O(k^{n-1}).\end{align*} The term $\hat{a}_0$ is independent of choice of $D$. By \cite[Lemma 9]{LS}, the term $\hat{b}_0$ is constant outside a Zariski-closed subset of $|T|$. As such we can make the following definition.

\begin{definition}\label{twisteddf} Take $D\in |T|$ to be a general element. We define the \emph{twisted Donaldson-Futaki invariant} of a twisted test configuration $(\scX,\scL,\scT)$ to be $$\DF(\scX,\scL,\scT) = \frac{b_0a_1 - b_1a_0}{a_0} + \frac{\hat{b}_0 a_0 - b_0 \hat{a}_0}{a_0}.$$\end{definition}

\begin{remark}\label{entryremark} We make the following remarks on the definition of the twisted Donaldson-Futaki invariant.
\begin{itemize} 
\item We show in Lemma \ref{linearitylemma} that the twisted term $\frac{\hat{b}_0 a_0 - b_0 \hat{a}_0}{a_0}$ associated to the line bundle $T$ is linear in $T$. As any line bundle can be written as the difference between two very ample line bundles, Definition \ref{twisteddf} makes sense for an arbitrary line bundle. When $T$ is semi-positive, writing $T=H_1\otimes H_2^{-1}$, a twisted test configuration requires lifting of the $\C^*$-action to line bundles on $\scX$ corresponding to \emph{both} $H_1$ and $H_2$, for some choice of $H_1,H_2$. However the Donaldson-Futaki invariant itself is independent of all choices, as we see in Lemma \ref{linearitylemma}.

\item While the definition of a twisted test configuration requires a lifting of the $\C^*$-action on $\scX$ to $\scT$, the twisted Donaldson-Futaki invariant only depends on the linear system $|T|$, and can be defined even if the $\C^*$-action does not lift to any line bundle on $\scX$ corresponding to $T$. In Proposition \ref{approximation} we show that one can approximate an arbitrary test configuration, whose action may not lift to any line bundle on $\scX$ corresponding to $T$, by twisted test configurations with arbitrarily close Donaldson-Futaki invariant. 

\item Test configurations can be thought of as geometrisations of the one-parameter subgroups that appear when using the Hilbert-Mumford criterion to check stability in geometric invariant theory.

\item In the untwisted case, the Donaldson-Futaki invariant is given as $\frac{b_0a_1 - b_1a_0}{a_0}$. The extra term $\frac{\hat{b}_0 a_0 - b_0 \hat{a}_0}{a_0}$ in the definition of the twisted Donaldson-Futaki invariant appears in the study of the J-flow, see \cite{LS}. \end{itemize}\end{remark}

We define the \emph{minimum norm} of a test configuration as follows. 

\begin{definition}\label{minnorm} Let $(\scX,\scL)$ be a test configuration. Assume the central fibre splits into irreducible components $\scX_{0,j}$, which by flatness must have dimension $n$. The $\C^*$-action fixes each component, hence we have a $\C^*$-action on $H^0(\scX_{0,j}, \scL_{0,j})$. Denote by $a_{0,j},b_{0,j}$ the leading terms of the corresponding Hilbert and weight polynomials. Let $\lambda_j$ be the minimum weight of the $\C^*$-action on the reduced support of the central fibre $X_0$. We define the \emph{minimum norm} of a test configuration to be $$\|\scX\|_m = \sum_j (b_{0,j} - \lambda_j a_{0,j}).$$

There is another definition of the minimum norm, which is also useful in practice. By scaling if necessary, assume $L$ is very ample. Let $D\in |L|$ be a divisor, and as above, denote the corresponding Hilbert and weight polynomials arising from the test configuration denoted by \begin{align*}\tilde{h}(k)&=\dim H^0(D,L|_D^k) = \tilde{a}_{0,D}k^{n-1}+O(k^{n-2}), \\ \tilde{w}(k)&=\wt H^0(\scD_0,\scL_0|_{\scD_0}^k)= \tilde{b}_{0,D}k^{n}+O(k^{n-1}). \end{align*} The term $\tilde{b}_{0,D}$ is constant outside a Zariski closed subset of $|D|$ \cite[Lemma 9]{LS}, define $\tilde{b}_0$ to equal this general value. By Remark \ref{equivalentminnorm}, we have \begin{equation}\|\scX\|_m =  \frac{\tilde{b}_0a_0 - b_0\tilde{a}_0}{a_0}.\end{equation}

 \end{definition}

Note that in the above definition we have $a_0 = \sum_j a_{0,j}$ and $b_0 = \sum_j b_{0,j}$, which we see explicitly in the proof of Theorem \ref{trivialitytheorem}. Remark also that if one scales $L->L^r$, and $\scL->\scL^r$, then the minimum scales by a factor of $r^{n+1}$. We will also later need to make use of the $L^2$\emph{-norm} of a test configuration. 

\begin{definition} By general theory, the trace of the square of the weights of the $\C^*$-action on $H^0(\scX_0, \scL_0^k)$ is a polynomial of degree $n+2$ for $k\gg 0$. Denoting this polynomial by $\tr(A_k^2) = d_0k^{n+2} + O(k^{n+1})$, we define the $L^2$\emph{-norm} of a test configuration to be $$\|\scX\|_2 = d_0 - \frac{b_0^2}{a_0}.$$ \end{definition}

We show in Theorem \ref{trivialitytheorem} that the condition $\|\scX\|_m=0$ is equivalent to the more familiar condition that the $L^2$-norm $\|\scX\|_2$ of the test configuration is zero. Moreover we show that this is equivalent to the test configuration being almost trivial, see Definition \ref{almosttrivial}. Finally we can define the notions of stability relevant to us.

\begin{definition} We say that $(X,L,T)$ is 
\begin{itemize}
\item\emph{twisted K-semistable} if $\DF(\scX,\scL,\scT) \geq0$ for all twisted test configurations $(\scX,\scL,\scT)$,
\item \emph{twisted K-stable} if $\DF(\scX,\scL,\scT) >0$ for all twisted test configurations $(\scX,\scL,\scT)$ with $\|\scX\|_m>0$,
\item \emph{uniformly twisted K-stable with respect to the minimum norm} if $$\DF(\scX,\scL,\scT)\geq\epsilon \|\scX\|_m$$ for all such $(\scX,\scL,\scT)$, and for some $\epsilon$ depending only on $(X,L)$.

\end{itemize} \end{definition}

\begin{remark} The definition of K-stability makes sense when $L$ and $T$ are $\Q$-line bundles, and our results hold in that generality. \end{remark}

We also recall the definition of (untwisted) K-stability.

\begin{definition}\label{untwistedkstab} When $T=\scO_X$, a twisted test configuration is called a test configuration. That is, the $\C^*$-action on $\scX$ lifts to $\scL$ but possibly no other line bundles on $\scX$. Setting $T=\scO_X$ in the definition of the twisted Donaldson-Futaki invariant, i.e. $$DF(\scX,\scL) = \frac{b_0a_1 - b_1a_0}{a_0},$$ we say a variety is \emph{K-stable} if for all test configurations with $\|\scX\|_m>0$ we have $\DF(\scX,\scL)>0$, with K-semistability and uniform K-stability defined similarly.\end{definition}

\subsection{Lower bounds on the twisted Calabi functional}

Let $X$ be a smooth complex projective variety with line bundles $L$ and $T$ . Take $h_L$ and $h_T$ be Hermitian metrics on $L$ and $T$ respectively with curvature forms $\omega$ and $\frac{1}{2}\alpha$. We assume throughout that $\omega$ is positive, so $L$ is ample, and $\alpha$ is semi-positive. We are interested in solutions of the \emph{twisted constant scalar curvature equation} $$S(\omega) - \Lambda_{\omega} \alpha = C_{\alpha}.$$ Here we define the contraction by $(\Lambda_{\omega} \alpha) \omega^n = n\alpha\wedge\omega^{n-1}$.  The topological constant $C_{\alpha}$ is given as $$C_{\alpha} = n\frac{(-K_X-2T).L^{n-1}}{L^n} = n\frac{\int_X (c_1(X)-2c_1(T)).c_1(L)^{n-1}}{c_1(L)^n}.$$ We can now state our main result.

\begin{theorem}\label{stabilityofcscK} Suppose $(X,L,T)$ admits a twisted cscK metric. \begin{itemize} \item[(i)] If $T$ is ample, then $(X,L,T)$ is uniformly twisted K-stable with respect to the minimum norm. \item[(ii)] If $T$ is semi-positive, then $(X,L,T)$ is twisted K-semistable.\end{itemize} \end{theorem}

Our strategy to prove Theorem \ref{stabilityofcscK} is to first prove the following lower bound on the twisted Calabi functional, which is the twisted analogue of \cite[Theorem 2]{D2}. 

\begin{proposition}\label{calabifunctional} $$\inf_{(\omega, \alpha)} \int_X (S(\omega) - \Lambda_{\omega} \alpha - C_{\alpha})^2\frac{\omega^n}{n!} \geq \sup_{(\scX,\scL,\scT)} -2\frac{\DF(\scX,\scL,\scT)}{\|\scX\|_{2}}.$$ \end{proposition} An immediate consequence of Proposition \ref{calabifunctional} is the following proof of part $(ii)$ of Theorem \ref{stabilityofcscK}. 

%\begin{corollary} Suppose $(X,L,T)$ admits a twisted cscK metric, with $T$ semi-positive. Then $(X,L,T)$ is twisted K-semistable. \end{corollary}

\begin{proof}[Proof of Theorem \ref{stabilityofcscK} (ii)] In this case the Calabi functional takes the value zero when $\omega$ is the twisted cscK metric. Proposition \ref{calabifunctional} then implies $\DF(\scX,\scL,\scT)\geq 0$ for all twisted test configurations $(\scX,\scL,\scT)$, i.e. $(X,L,T)$ is twisted K-semistable. \end{proof}

The metrics $h_L,h_T$ induce a metric $h_L^k\otimes h_T^{-1}$ on $L^k\otimes T^{-1}$. We give the sequence of vector spaces $H^0(X,L^k\otimes T^{-1})$ an $L^2$-inner product by $$\langle s, t \rangle = \int_X (s,t)_{h_L^k\otimes h_T^{-1}} \frac{(k\omega)^n}{n!}.$$ 

\begin{definition} Choose an $L^2$-orthonormal basis $s_0,\hdots,s_{N_k}$ of $H^0(X,L^k\otimes T^{-1})$. We define the \emph{Bergman kernel} associated to $h_L,h_T$ to be $$\rho_{k} = \sum_{i=0}^{N_k} |s_i|^2_{h_L^k\otimes h_T^{-1}}.$$ \end{definition}

The key result regarding asymptotics of the Bergman kernel is as follows.

\begin{theorem}\cite[Theorem 4.1.2]{MM}\label{bergmankernel} As $k\to\infty$ the Bergman kernel admits a $C^{\infty}$ expansion \begin{equation}\label{bergmanformula} \rho_k = 1 + \frac{S(\omega) - \Lambda_{\omega} \alpha}{2}k^{-1} + O(k^{-2}).\end{equation}\end{theorem}

The reason to make this definition is the link with embeddings in projective space. Indeed, a basis of $H^0(X,L^k\otimes T^{-1})$ gives an embedding of $X$  by $x\to [s_0(x):\hdots :s_{N_k}(x)]$. Denote this map by $\varphi_k: X \hookrightarrow \pr(H^0(X,L^k\otimes T^{-1}))$, and let $\omega_{FS}$ be the corresponding Fubini-Study metric.

\begin{lemma}\label{curvatureforms} The Bergman kernel $\rho_k$ satisfies \begin{equation}\label{bergmankernelexpansion}\varphi_k^*(\omega_{FS}) - k\omega +\frac{1}{2}\alpha = i\partial \bar{\partial}\log \rho_k.\end{equation} \end{lemma}

\begin{proof} On the open subset of $X$ on which $s_0\neq 0$, we have $$ \varphi_k^*\omega_{FS} = i\partial\bar{\partial} \log \left(1 + |\frac{s_1}{s_0}|^2+\hdots + |\frac{s_N}{s_0}|^2\right).$$ Here we have used the fact that the expansion (\ref{bergmanformula}) holds in $C^2$. The quotients satisfy $|\frac{s_1}{s_0}|^2 = \frac{|s_1|^2_{h_L^k\otimes h_T^{-1}}}{|s_0|^2_{h_L^k\otimes h_T^{-1}}}$. Hence 

\begin{align*}\varphi_k^*\omega_{FS} &= i\partial\bar{\partial} \log \left(1 + \frac{|s_1|^2_{h_L^k\otimes h_T^{-1}}}{|s_0|^2_{h_L^k\otimes h_T^{-1}}}+\hdots + \frac{|s_N|^2_{h_L^k\otimes h_T^{-1}}}{|s_0|^2_{h_L^k\otimes h_T^{-1}}}\right), \\ 
&= -i\partial\bar{\partial} \log |s_0|^2_{h_L^k\otimes h_T^{-1}} + i\partial\bar{\partial}\left(\log \sum |s_i|^2_{h_L^k\otimes h_T^{-1}}\right).\end{align*} The result follows as $i\partial\bar{\partial} \log |s_0|^2_{h_L^k\otimes h_T^{-1}} = k\omega - \frac{1}{2}\alpha$. \end{proof}

\begin{corollary}\label{volformexpansion} As $k\to \infty$, the curvature forms satisfy $$\varphi_k^*(\omega_{FS}) - k\omega + \frac{1}{2}\alpha = O(k^{-2}).$$ In particular, we have $$(\varphi_k^*\omega_{FS})^n+\frac{1}{2}\alpha\wedge(\varphi_k^*\omega_{FS})^{n-1}=(k\omega)^n+O(k^{n-2}).$$\end{corollary}

\begin{proof} The expansion of the curvature forms follows by the power series expansion of the logarithm function appearing in Lemma \ref{curvatureforms}. Taking the top exterior power we have $$(\varphi_k^*\omega_{FS})^n = (k\omega)^n - \frac{1}{2}\alpha\wedge(k\omega)^{n-1}+O(k^{n-2}),$$ using $$\alpha\wedge(\varphi_k^*\omega_{FS})^{n-1} = \frac{1}{2}\alpha\wedge(k\omega)^{n-1}+O(k^{n-2})$$ gives the result. \end{proof}

For a variety $X$ embedded in projective space $\varphi_k: X \hookrightarrow \pr^{N_k}$ and a semi-positive $(1,1)$-form $\alpha$ on $X$, we define a matrix $M(\varphi_k)$ to be $$M(\varphi_k)_{ij} =  \int_{X} \frac{z_i\bar{z}_j}{|z|^2}\frac{\omega_{FS}^n}{n!} + \frac{1}{2}\int_{X}\frac{z_i\bar{z_j}}{|z|^2}\frac{\alpha\wedge \omega_{FS}^{n-1}}{(n-1)!}.$$ Denote by $\underline{M}(\varphi_k)$ be the trace-free part of $M(\varphi_k)$, and use the norm $\|M\|= Tr(MM^*)$. Using this norm, we have the following lower bound on the Calabi functional, entirely analogously to \cite[Proposition 1]{D2}.

\begin{lemma}\label{calabibound} $\|\underline{M}(\varphi_k)\| \leq \frac{k^{n/2-1}}{2}\|S(\omega) - \Lambda_{\omega} \alpha - C_{\alpha}\|_{L^2} + O(k^{n/2-2})$.\end{lemma}

\begin{proof} By a unitary transformation, we assume $M(\varphi_k)$ is diagonal. Pulling back the integrals to $X$ we have $$ M(\varphi_k)_{ii} = \int_{X} \frac{|s_i|^2_{\varphi_k^*(h_{FS})}}{\sum^{N_k}_{j=0} |s_j|^2_{\varphi_k^*(h_{FS})}} \frac{(\varphi_k^*\omega_{FS})^n}{n!} + \frac{1}{2}\int_{X} \frac{|s_i|^2_{\varphi_k^*(h_{FS})}}{\sum^{N_k}_{j=0} |s_j|^2_{\varphi_k^*(h_{FS})}}\frac{\alpha\wedge(\varphi_k^*\omega_{FS})^{n-1}}{(n-1)!}.$$ By the definition of the Bergman kernel and using Corollary \ref{volformexpansion} this becomes $$ M(\varphi_k)_{ii} = \int_{X} \frac{|s_i|^2_{h_L^k\otimes h_T^{-1}}}{\rho_k} \frac{(k\omega)^n}{n!}+O(k^{n-2}).$$ By Theorem \ref{bergmankernel} we see that $$M(\varphi_k)_{ii} = \int_{X} |s_i|^2_{h_L^k\otimes h_T^{-1}}\left(1-\frac{S(\omega) - \Lambda_{\omega} \alpha}{2}k^{-1}\right) \frac{(k\omega)^n}{n!} + O(k^{n-2}),$$ which in turn using $L^2$-orthonormality of $s_i$ gives $$M(\varphi_k)_{ii} = 1 - k^{-1}\int_{X} |s_i|^2_{h_L^k\otimes h_T^{-1}}\left(\frac{S(\omega) - \Lambda_{\omega} \alpha}{2}\right) \frac{(k\omega)^n}{n!} + O(k^{n-2}).$$

We now calculate the tracefree part of $M(\varphi_k)$. Firstly the rank of $M(\varphi_k)$ is given as $N_k+1 = \dim H^0(X,L^k\otimes T^{-1}).$ The trace is given as \begin{align*} Tr(M(\varphi_k)) &= N_k+1 - k^{-1}\int_{X} \rho_k \left(\frac{S(\omega) - \Lambda_{\omega} \alpha}{2}\right) \frac{(k\omega)^n}{n!} + O(k^{n-2}), \\  &= N_k+1 - k^{-1}\int_{X}\left( \frac{S(\omega) - \Lambda_{\omega} \alpha}{2} \right)\frac{(k\omega)^n}{n!} + O(k^{n-2}).\end{align*} Therefore $$\frac{Tr(M(\varphi_k))}{rk(M(\varphi_k))} = 1 - \frac{k^{-1}}{2}C_{\alpha} + O(k^{-2}).$$ Using this we see that the tracefree part of $M(\varphi_k)$ is given as $$\underline{M}(\varphi_k)_{ii} = \frac{k^{-1}}{2} \int_{X} |s_i|^2_{h_L^k\otimes h_T^{-1}}(C_{\alpha} -S(\omega) + \Lambda_{\omega} \alpha) \frac{(k\omega)^n}{n!} + O(k^{n-2}).$$

%clean the next inequality
By Cauchy-Schwarz and $L^2$-orthonormality of the $s_i$ we have \begin{align*} |\underline{M}(\varphi_k)_{ii}|^2 &\leq \frac{k^{-2}}{4} \int_{X} |s_i|^2_{h_L^k\otimes h_T^{-1}}\frac{(k\omega)^n}{n!}\times\\ &\times\int_{X} |s_i|^2_{h_L^k\otimes h_T^{-1}}\left(C_{\alpha} - \frac{S(\omega) - \Lambda_{\omega} \alpha}{2}\right)^2 \frac{(k\omega)^n}{n!} + O(k^{n-2}),\\
 &= \frac{k^{-2}}{4} \int_{X} |s_i|^2_{h_L^k\otimes h_T^{-1}}\left(C_{\alpha} - \frac{S(\omega) - \Lambda_{\omega} \alpha}{2}\right)^2 \frac{(k\omega)^n}{n!} + O(k^{n-2}).\end{align*} Summing over the basis vectors and once again using the expansion of the Bergman kernel we get \begin{align*} \|\underline{M}(\varphi_k)\|^2 &\leq \frac{k^{-2}}{4} \int_{X} \rho_k\left(C_{\alpha} - \frac{S(\omega) - \Lambda_{\omega} \alpha}{2}\right)^2 \frac{(k\omega)^n}{n!} + O(k^{n-2}), \\ & =  \frac{k^{n-2}}{4} \int_{X}\left(C_{\alpha} - \frac{S(\omega) - \Lambda_{\omega} \alpha}{2}\right)^2 \frac{\omega^n}{n!} + O(k^{n-2}).\end{align*} Finally taking the square root of both sides gives the required result.\end{proof}

\begin{remark} In \cite{JK}, Keller shows that the existence of \emph{balanced embedding}, i.e. an embedding with $\|\underline{M}(\varphi_k)\| = 0$, is equivalent to the Bergman kernel as defined above being constant. This gives partial motivation for our embedding of $X$. \end{remark}

\subsection{The argument in a fixed projective space}\label{fixedproj}

Next we give a more algebraic lower bound for norm of the matrix $\|\underline{M}(\varphi_k(X))\|$. Let $\lambda(t): \C^* \hookrightarrow GL(N_k+1)$ be a one-parameter subgroup, where $N_k+1 = \dim H^0(X,L^k\otimes T^{-1})$. This $\C^*$-action contains $S^1$ as a subgroup; assume this action lies inside $U(N_k+1)$. Thus the $S^1$-action is given by $t^B$ for some Hermitian matrix $B$ with integer eigenvalues. With respect to the Fubini-Study metric, this action has an associated Hamiltonian function $$h_B = -\frac{B_{jk}z_j\bar{z}_k}{|z|^2}.$$ Here $B$ acts on $\pr^N$ by the dual action, which leads to the minus sign. Denote $\varphi_k^t = \lambda(t) \circ \varphi_k$, which is therefore a map $X \to \pr^{N_k}$. We define a function $f(t)$ by $$f(t) = -Tr(\underline{BM}(\varphi_k^t))=-Tr(\underline{B}M(\varphi_k^t)).$$ In terms of integrals this is \begin{align*}f(t) = \int_{\varphi_k^t(X)} h_B \frac{\omega_{FS}^n}{n!} +& \frac{1}{2}\int_{\varphi_k^t(X)} h_B \frac{\alpha\wedge\omega^{n-1}_{FS}}{(n-1)!}+  \\ & +\frac{\tr(B_k)}{N_k+1}\left(\int_{\varphi_k^t(X)}\frac{\omega_{FS}^n}{n!} + \frac{1}{2}\int_{\varphi_k^t(X)} \frac{\alpha\wedge\omega^{n-1}_{FS}}{(n-1)!}\right).\end{align*} What we need is that $f(t)$ is non-decreasing for positive real $t$. 

\begin{lemma}\label{sumconvex} The function $f(t) = -Tr(\underline{BM}(\varphi_k^t))$ satisfies $f'(t) \geq 0$ for $t\in \R_{>0}$. \end{lemma}

\begin{proof} It suffices to calculate the derivative at $t=0$. The function $f(t)$ is given as a sum of two functions, write $f(t) = f_1(t) + f_2(t)$ where \begin{align*} &f_1(t) = \int_{\varphi_k^t(X)} h_B \frac{\omega_{FS}^n}{n!} + \frac{\tr(B_k)}{N_k+1}\left(\int_{\varphi_k^t(X)}\frac{\omega_{FS}^n}{n!}\right),\\ &f_2(t) = \frac{1}{2}\int_{\varphi_k^t(X)} h_B \frac{\alpha\wedge\omega^{n-1}_{FS}}{(n-1)!} + \frac{\tr(B_k)}{N_k+1}\left(\frac{1}{2}\int_{\varphi_k^t(X)} \frac{\alpha\wedge\omega^{n-1}_{FS}}{(n-1)!}\right).\end{align*} Noting that $ Tr(\underline{BM}(\varphi_k^t))$ is a Hamiltonian for the corresponding vector field with respect to the Fubini-Study metric, Donaldson \cite[Proposition 2]{D2}  shows that $f_1'(t)  \geq 0.$ In particular note that we are in a fixed projective space, with a $\C^*$-action, and a Hamiltonian with respect to the Fubini-Study metric, so Donaldson's argument applies. 

Similarly, by a  differential-geometric calculation, Lejmi-Sz\'ekelyhidi \cite[Lemma 8]{LS} show that $f_2'(t) \geq 0$ provided $\alpha$ is semi-positive. \end{proof}

\begin{remark} We use here in an essential way that $\alpha$ is semi-positive. This means that $f(t)$ is a sum of two convex functions, which must therefore be convex. For example, there is \emph{a priori} no link between solutions to $$S(\omega) + \Lambda_{\omega}\alpha = C_{\alpha}$$ and algebraic geometry, since the corresponding function $f(t)$ is no longer automatically convex. \end{remark}

\begin{remark} There is an interesting analogy here with Stoppa's proof that the existence of a twisted cscK metric implies semistability with respect to test configurations arising from deformation to the normal cone \cite{JS}. Indeed, his proof involves an asymptotic expansion of the twisted Mabuchi functional, which is the sum of the usual Mabuchi functional and another functional arising from the study of the J-flow. Just as in Lemma \ref{sumconvex}, it is important in his approach that both functionals involved are convex.  \end{remark}

\begin{corollary}\label{limitto0} $\|\underline{B}\| \|\underline{M}(\varphi_k) \| \geq \lim_{t\to 0}f(t).$ \end{corollary}

\begin{proof} Since $f(t)$ is non-decreasing we have $$f(1) = -Tr(\underline{BM}(\varphi_k)) \geq \lim_{t\to 0}f(t).$$ Cauchy-Schwarz gives $$\|\underline{B}\|\|\underline{M}(\varphi_k)\| \geq \|\underline{BM}(\varphi_k)\| = -Tr(\underline{BM}(\varphi_k)),$$ from which the result follows. \end{proof}

In order to give an algebro-geometric interpretation of $\lim_{t\to 0}f(t)$, we use the following Lemma to identify the form $\alpha$ with an average over the currents in the linear system associated with $T$. 

\begin{lemma}\cite[Lemma 9]{LS}\label{currentsasaverage} Let $T$ be a very ample line bundle with $D\in |T|$. Denoting by $D_0$ the flat limit of $D$ under the $\C^*$-action with corresponding cycle $|X_0|$, the integral \begin{equation}\label{general} \int_{|D_0|}h_B \frac{\omega^{n-1}_{FS}}{{(n-1)!}}\end{equation}  is constant outside a Zariski closed subset of $|T|$. We say that $D$ is \emph{general} if it is chosen such that the integral in equation (\ref{general}) takes its general value. Then, for general $D$, we have $$\lim_{t\to 0} \frac{1}{2}\int_{\varphi_k^t(X)} h_B \frac{\alpha\wedge\omega^{n-1}_{FS}}{(n-1)!} = \int_{|D_0|}h_B \frac{\omega^{n-1}_{FS}}{{(n-1)!}}.$$  \end{lemma}

\begin{remark}\label{linearityremark} This Lemma implies that the integral we are interested in is a topological invariant, which does not depend on the choice of $\frac{1}{2}\alpha\in c_1(T)$. We emphasise that what it means $D \in |T|$ to be general depends on the $\C^*$-action on projective space.  \end{remark}

When $T$ is not very ample, we can still utilise Lemma \ref{currentsasaverage} as follows.

\begin{lemma}\label{sposcurrents} Suppose $\frac{1}{2}\alpha\in c_1(T)$ is semi-positive. Write $T = H_1\otimes H_2^{-1}$ for very ample line bundles $H_1,H_2$, and let $F_1\in |H_1|, F_2\in |H_2|$ be divisors. Denote by $F_{1,0}$ and $F_{2,0}$ the flat limits of $F_1,F_2$ respectively under the $\C^*$-action and $|F_{1,0}|,|F_{2,0}|$ the corresponding cycles. Then the integral \begin{equation*} \int_{|F_{1,0}|}h_B \frac{\omega^{n-1}_{FS}}{{(n-1)!}}-\int_{|F_{2,0}|}h_B \frac{\omega^{n-1}_{FS}}{{(n-1)!}}\end{equation*} is constant outside Zariski closed subsets of $|H_1|$ and $|H_2|$. Moreover, we have $$\lim_{t\to 0} \frac{1}{2}\int_{\varphi_k^t(X)} h_B \frac{\alpha\wedge\omega^{n-1}_{FS}}{(n-1)!} = \int_{|F_{1,0}|}h_B \frac{\omega^{n-1}_{FS}}{{(n-1)!}}-\int_{|F_{2,0}|}h_B \frac{\omega^{n-1}_{FS}}{{(n-1)!}},$$ where $F_1,F_2$ are chosen such that the integral on the right hand side takes its general value. We say in this case that $F_{1,0}$ and $F_{2,0}$ are \emph{general}. \end{lemma}

\begin{proof} We first claim that is possible to write $\alpha = \eta_{1} - \eta_2$ for $\eta_i\in c_1(H_i)$ smooth positive $(1,1)$-forms. Indeed, take arbitrary positive $\zeta_1 \in c_1(H_1), \zeta_2 \in c_1(H_2)$. Then $\alpha$ and $\zeta_1 - \zeta_2$ are cohomologous, by the $\partial \bar{\partial}$-lemma we can write $$\alpha = \zeta_1 - \zeta_2 + i\partial\bar{\partial} \psi$$ for some smooth function $\psi$. Recall $\alpha$ is semi-positive, so $\alpha + \zeta_2\in c_1(T\otimes H_2)$ is positive. In particular, $\zeta_1 + i\partial\bar{\partial} \psi \in c_1(H_1)$ is also positive. Letting $\eta_1 = \zeta_1 + i\partial\bar{\partial}\psi$ and $\eta_2 = \zeta_2$ gives the required form. 

Note that $$\lim_{t\to 0} \frac{1}{2}\int_{\varphi_k^t(X)} h_B \frac{\alpha\wedge\omega^{n-1}_{FS}}{(n-1)!}=\lim_{t\to 0}\frac{1}{2}\left( \int_{\varphi_k^t(X)} h_B \frac{\eta_1\wedge\omega^{n-1}_{FS}}{(n-1)!} -\int_{\varphi_k^t(X)} h_B \frac{\eta_2\wedge\omega^{n-1}_{FS}}{(n-1)!} \right).$$ Lemma \ref{currentsasaverage} implies that for $i=1,2$ we have $$\lim_{t\to 0}\frac{1}{2} \int_{\varphi_k^t(X)} h_B \frac{\eta_1\wedge\omega^{n-1}_{FS}}{(n-1)!} = \int_{|F_{i,0}|}h_B \frac{\omega^{n-1}_{FS}}{{(n-1)!}},$$ from which the result follows.\end{proof}

\begin{remark}Similarly $\lim_{t\to 0} X$ is well defined as a \emph{cycle}, which we denote by $|X_0|$. Even when $T$ is just semi-positive, we still write $|D_0|$ for the cycle $|F_{1,0}|-|F_{2,0}|$ for general $F_{1,0},F_{2,0}$ in the sense of Lemma \ref{sposcurrents}.\end{remark}

Since $\lim_{t\to 0} X = |X_0|$, the corresponding integrals converge as $t\to 0$, giving the following. 

\begin{corollary} Let $D\in|T|$ be a general element. Then \begin{equation}\label{limitint}\lim_{t\to 0}f(t) = \int_{|X_0|} h_B \frac{\omega_{FS}^n}{n!} + \int_{|D_0|} h_B \frac{\omega^{n-1}_{FS}}{(n-1)!} + \frac{\tr(B_k)}{N_k+1}\left(\int_{|X_0|}\frac{\omega_{FS}^n}{n!} +\int_{|D_0|} \frac{\omega^{n-1}_{FS}}{(n-1)!}\right).\end{equation} Here if the central fibre has non-reduced structure, the integral is calculated with appropriate multiplicity. In particular, the limit is independent of choice of $\frac{1}{2}\alpha\in c_1(T)$.\end{corollary}

\begin{remark}\label{fsabuse} The integral in equation (\ref{limitint}) depends on $k$, as we are abusing notation by denoting by $\omega_{FS}$ the Fubini-Study metric arising from the embedding of $X$ into projective space using global sections of $L^k\otimes T^{-1}$. \end{remark}

\subsection{Test configurations and the asymptotic argument}

Corollary \ref{limitto0} gives the lower bound $$\|\underline{M}(\varphi_k) \| \geq \frac{\lim_{t\to 0}f(t)}{\|\underline{B}\|}.$$ However, Proposition \ref{calabibound} gives a lower bound for the norm of $\underline{M}(\varphi_k)$ only to high order in $k$. The following Proposition shows that twisted test configurations embed equivariantly into projective space, so that we are in the situation of Corollary \ref{limitto0}. 

\begin{proposition}\label{equivariantembedding}\cite[Lemma 2]{D2} Let $E\to\C$ be a $\C^*$-equivariant vector bundle over $\C$ with fibre $E_t$. Then for all Hermitian metrics $h_1$ on the fibre $E_1$, there exists an equivariant trivialisation $E\cong E_0\times \C$ taking $h_1$ to a Hermitian metric on $h_0$ on $E_0$ preserved by the action of $S^1\subset \C^*$. \end{proposition}

Applying this to $E = \pi_*(\scL^k\otimes \scT^{-1})$ gives an equivariant embedding of $\scX$ into projective space $\pr(H^0(\C,\pi_*(\scL_0^k\otimes \scT_0^{-1}))^*)\times \C\cong \pr(H^0(X,L^k\otimes T^{-1})^*)\times \C$ such that the induced $S^1$-action on $\pr(H^0(X,L^k\otimes T^{-1}))$ is unitary. Note that the proposition gives another natural embedding of $\scX$ into projective space through global sections of $L^k$, we will later make use of both embeddings.

We now return to the situation of a $\C^*$-action in a fixed projective space, so we have a decreasing function $f(t)$ with limit \begin{equation}\label{limitoft}\lim_{t\to 0}f(t) = \int_{|X_0|} h_B \frac{\omega_{FS}^n}{n!} +\int_{|D_0|} h_B \frac{\omega^{n-1}_{FS}}{(n-1)!} + \frac{\tr(B_k)}{N_k+1}\left(\int_{|X_0|}\frac{\omega_{FS}^n}{n!} + \int_{|D_0|} \frac{\omega^{n-1}_{FS}}{(n-1)!}\right),\end{equation} where $D$ was assumed to be general. Recall by Remark \ref{fsabuse} that this Fubini-Study metric, and hence this integral, depends on $k$.

\begin{proposition}\label{limitisdf} The limit $\lim_{t\to 0}f(t)$ is given as $$\lim_{t\to 0}f(t) = (-\DF(\scX,\scL,\scT))k^n + O(k^{n-1}).$$ \end{proposition}

%Note that the Donaldson-Futaki invariant on the right hand side of the equation is calculated using weights $L^k$, while the function $f(t)$ itself is calculated in the embedding by global sections of $L^k\otimes T^{-1}$. 

Following Donaldson \cite[Proposition 3]{D2}, instead of working equivariantly on the central fibre $X_0$, we calculate the needed quantities non-equivariantly on a scheme of one dimension higher. 

Denote by $\scO_{\pr^1}(1)^*$ the principal $\C^*$-bundle over $\pr^1$ given as the complement of the zero section. Since $(X_0,\scL_0,\scT_0)$ admits a $\C^*$-action, we can form the associated $(X_0,\scL_0,\scT_0)$ bundle $$(\scY,\scH_L,\scH_T) = \scO_{\pr^1}(1)^* \times_{\C^*} (X_0,\scL_0,\scT_0).$$ Note that $\scY$ is \emph{intrinsically} constructed from $\scX_0$. Moreover, given any divisor $D_0\subset \scX_0$, we get a corresponding scheme $\scZ\subset\scY$ which fibres over $\pr^1$. 

\begin{lemma}\label{dgtoag} The following Chern-Weil type formulae hold.

\begin{itemize} 
\item[(i)] The lead term in the Hilbert polynomial of $(X,L)$ is given asymptotically as $$\int_{|X_0|}\frac{\omega_{FS}^n}{n!} +\int_{|D_0|} \frac{\omega^{n-1}_{FS}}{(n-1)!} = a_0k^n + O(k^{n-2}).$$

\item[(ii)] The weight polynomial $\wt(H^0(X_0,\scL_0^k))$ is given as $$\wt(H^0(X_0,\scL_0^k)) = \chi(\scY,\scH_L^k) - \chi(X_0,\scL_0^k),$$ and a similar formula holds for $\wt(H^0(X_0,\scL_0^k\otimes \scT_0^{-1}))$.

\item[(iii)] The Hilbert polynomials are related by $$\dim H^0(X,L^k) = \dim H^0(X,L^k\otimes T^{-1}) + \hat{a}_0k^{n-1} + O(k^{n-2}).$$

\item[(iv)] The weight polynomials are related by $$\wt(H^0(X_0,\scL_0^k)) = \wt(H^0(X_0,\scL_0^k\otimes\scT_0^{-1})) + k^n\int_{\scY}\frac{c_1(\scH_L)^n.c_1(\scH_T)}{n!} + O(k^{n-1}).$$

\item[(v)] The remaining terms in the formula (\ref{limitoft}) can be calculated as $$\int_{|X_0|} h_B \frac{\omega_{FS}^n}{n!} + \int_{|D_0|} h_B \frac{\omega_{FS}^{n-1}}{(n-1)!} = -b_0k^{n+1} + ck^n,$$ where $$c =-\int_{\scZ} \frac{c_1(\scH_L)^{n}}{n!} + \int_\scY \frac{c_1(\scH_L^n).c_1(\scH_T)}{n!}.$$

\item[(vi)] The asymptotics of the square of the weight are given as $$\|\underline{B}\| = \|\scX\|_{2}k^{n+2} + O(k^{n+1}).$$ 
\end{itemize}

\end{lemma}

%\begin{remark} In fact part $(v)$ of the previous Lemma is the reason we use the embedding using sections of $L^k$ in our definition of the twisted Donaldson-Futaki invariant, it would be interesting to know if the term $c$ has a natural algebro-geometric description. \end{remark}

Postponing the proof of Lemma \ref{dgtoag} for the moment,  we complete the proof of Proposition \ref{limitisdf}.

\begin{proof}[Proof of Proposition \ref{limitisdf}] The quantity we wish to calculate is $$\lim_{t\to 0}f(t) = \int_{|X_0|} h_B \frac{\omega_{FS}^n}{n!} + \int_{|D_0|} h_B \frac{\omega^{n-1}_{FS}}{(n-1)!} + \frac{\tr(B_k)}{N_k+1}\left(\int_{|X_0|}\frac{\omega_{FS}^n}{n!} + \int_{|D_0|} \frac{\omega^{n-1}_{FS}}{(n-1)!}\right).$$ By Lemma \ref{dgtoag} $(v)$, this is equal to \begin{align*} \lim_{t\to 0}f(t) = -b_0k^{n+1} - \int_{\scZ}& \left(\frac{c_1(\scH_L)^{n}}{n!} - \int_\scY \frac{c_1(\scH_L^n).c_1(\scH_T)}{n!}\right)k^n + \\ &+ \frac{\wt( H^0(\scX_0,\scL_0^k\otimes\scT_0^{-1})) }{\dim H^0(X,L^k \otimes T^{-1}) }a_0+ O(k^{n-1}).\end{align*} 

Note that $\int_{\scZ}\frac{c_1(\scH_L)^{n}}{n!}  = \hat{b}_0k^n$. By parts $(iii),(iv)$ of Lemma \ref{dgtoag}, we have expansions \begin{align*} 
& \dim H^0(X,L^k\otimes T^{-1}) = a_0k^n + (a_1-\hat{a}_0)k^{n-1} + O(k^{n-2}), \\
& \wt(H^0(\scX_0,\scL_0^k\otimes\scT_0^{-1})) = b_0k^{n+1} + \left(b_1 - \int_{\scY}\frac{c_1(\scH_L)^n.c_1(\scH_T)}{n!}\right)k^n + O(k^{n-1}). \end{align*}

This leads to \begin{align*} \lim_{t\to 0}f(t) = &-b_0k^{n+1} - \hat{b}_0k^{n} + k^n\int_\scY \frac{c_1(\scH_L)^n.c_1(\scH_T)}{n!} + \\ & + \frac{b_0k^{n+1} + \left(b_1 - \int_{\scY}\frac{c_1(\scH_L)^n.c_1(\scH_T)}{n!}\right)k^n + O(k^{n-1})}{a_0k^n + (a_1-\hat{a}_0)k^{n-1} + O(k^{n-2})} a_0+ O(k^{n-1}).\end{align*}

The integrals over $\scY$ of degree $k^n$ cancel so we are left with \begin{align*} \lim_{t\to 0}f(t) &= -b_0k^{n+1} - \hat{b}_0k^n + \frac{b_0k^{n+1}+b_1k^n+ O(k^{n-1})}{a_0k^n + (a_1-\hat{a}_0)k^{n-1} + O(k^{n-2})} a_0+ O(k^{n-1}), \\
&= k^n\left(\frac{a_0(b_1 - \hat{b}_0) - b_0(a_1-\hat{a}_0)}{a_0}\right) + O(k^{n-1}), \\
&= k^n(-\DF(\scX,\scL,\scT)) + O(k^{n-1}), \end{align*} as required.

\end{proof}

Summing up, we have the following proof of Proposition \ref{calabifunctional}.

\begin{proof}[Proof of Proposition \ref{calabifunctional}] Given arbitrary Hermitian metrics $h_L, h_T$, and any twisted test configuration $(\scX,\scL,\scT)$, we use Lemma \ref{equivariantembedding} to embed the test configuration into projective space by global sections of $L^k\otimes T^{-1}$. Applying Lemma \ref{calabibound} and the convexity property given in Corollary \ref{limitto0} we have \begin{align*} \frac{k^{n/2-1}}{2}\|S(\omega) - \Lambda_{\omega} \alpha - C_{\alpha}\|_{L^2} &\geq \|\underline{M}(\varphi_k)\| + O(k^{n/2-2}),\\ &\geq \frac{\lim_{t\to 0}f(t)}{\|\underline{B}\|}+ O(k^{n/2-2}).\end{align*} Finally, combining Proposition \ref{limitisdf} and Lemma \ref{dgtoag} $(vi)$ gives $$\|S(\omega) - \Lambda_{\omega} \alpha - C_{\alpha}\|_{L^2} \geq -2\frac{\DF(\scX,\scL,\scT)}{\|\scX\|_{2}}+ O(k^{-2-n/2}).$$ Taking the $k\to\infty$ limit provides the result. \end{proof}

We now return to the proof of Lemma \ref{dgtoag}.

\begin{proof}[Proof of Lemma \ref{dgtoag}] \ \\ \begin{itemize} 
\item[(i)] By Riemann-Roch and flatness, the leading term in the Hilbert polynomial is given as $$a_0 = \int_X\frac{c_1(L)^n}{n!} = \int_{|X_0|}\frac{c_1(\scL_0)^n}{n!}.$$ Since we have embedded $X\hookrightarrow \pr(H^0(X,L^k\otimes T^{-1}))$, the Fubini-Study metric satisfies $\omega_{FS} \in c_1(\scL_0^k \otimes \scT_0^{-1})$. As the integrals we wish to calculate are topological invariants, the quantity we wish to calculate is given as \begin{align*}k^n\int_{|X_0|} \frac{c_1(\scL_0)^n}{n!}& + k^{n-1}\int_{|D_0|} \frac{c_1(\scL_0)^{n-1}}{(n-1)!} - \\ &- k^{n-1}\int_{|X_0|} \frac{c_1(\scL_0)^{n-1}.c_1(\scT_0)}{n!} + O(k^{n-2}).\end{align*} Again by flatness this is equal to the corresponding quantity over any non-zero fibre. Hence the order $k^{n-1}$ terms cancel, as \begin{align*}\int_{|D_0|} \frac{c_1(\scL_0)^{n-1}}{(n-1)!} &- \int_{|X_0|} \frac{c_1(\scL_0)^{n-1}.c_1(\scT_0)}{n!} = \\ &= \int_{|D_t|} \frac{c_1(\scL_t)^{n-1}}{(n-1)!} - \int_{|X_t|} \frac{c_1(\scL_t)^{n-1}.c_1(\scT_t)}{n!} = 0,\end{align*} where the last equality follows as $D \in |T|$.

\item[(ii)] This is contained in \cite[Proposition 3]{D2}, we follow the exposition of Ross-Thomas \cite[Proposition 2.19]{RT3}. Denote by $\eta: \scY \to \pr^1$ the natural projection. By flatness, $\eta_*\scH_L^k$ is a vector bundle over $\pr^1$, which as such must split as a direct sum of line bundles. By construction we see that $\eta_*\scH_L^k \cong \scO_{\pr^1}(1)^* \times_{\C^*} H^0(\scX_0,\scL_0^k)$. A splitting of $H^0(\scX_0,\scL_0^k)$ into weight spaces corresponds to splitting $\eta_*\scH_L^k$ into line bundles. The total weight of the action on $H^0(\scX_0,\scL_0^k)$ is therefore equal to the first Chern class of $\eta_*\scH_L^k$; by Riemann-Roch this gives $$\wt(H^0(\scX_0,\scL_0^k)) = \chi(\pr^1,\eta_*\scH_L^k) - \rk(\eta_*\scH_L^k) = \chi(\scY,\scH_L^k) - \chi(X_0,\scL_0^k),$$ where the last equality follows by definition of the pushforward.

\item[(iii)] Since the equality we seek is additive in $T$, we assume that $T$ is very ample. Taking a divisor $D \in |T|$ leads to the restriction short exact sequence $$0 \to L^k\otimes T^{-1} \to L^k \to L|_D^k\to 0$$ The result follows by noting that $\chi(D,L|_{D}^k) = \hat{a}_0k^{n-1} + O(k^{n-2})$ and additivity of the Euler characteristic in short exact sequences.  

\item[(iv)] We now move to the relation between the weights. By $(ii)$, to prove $$\wt(H^0(X_0,\scL_0^k)) = \wt(H^0(X_0,\scL_0^k\otimes\scT_0^{-1})) + k^n\int_{\scY}\frac{c_1(\scH_L)^n.c_1(\scH_T)}{n!} + O(k^{n-1}),$$ it is equivalent to prove \begin{align*}\chi(\scY,\scH_L^k) -& \chi(X_0,\scL_0^k) = \chi(\scY,\scH_L^n\otimes\scH_T^{-1}) - \\ &- \chi(X_0,\scL_0^k\otimes\scT_0^{-1}) + k^n\int_{\scY}\frac{c_1(\scH_L)^n.c_1(\scH_T)}{n!} + O(k^{n-1}).\end{align*} By $(iii)$, the Euler characteristics on $X_0$ are equal to order $O(k^{n-1})$, so we wish to show $$\chi(\scY,\scH_L^k)  = \chi(\scY,\scH_L^k\otimes\scH_T^{-1})+k^n\int_{\scY}\frac{c_1(\scH_L)^n.c_1(\scH_T)}{n!}+ O(k^{n-1}).$$

Note that even if $\scT_0$ is ample, the associated bundle $\scH_T$ may not be effective. Write $\scH_T = H_1\otimes H_2^{-1}$, where $H_1$ and $H_2$ are very ample line bundles. Take a divisor $F_1\in|H_1|$. The corresponding restriction exact sequence is $$0 \to \scH_L^k\otimes H_1^{-1} \to \scH_L^k \to \scH_L|_{F_1}^k\to 0.$$ 

We can expand the Euler characteristic $\chi(F_,\scH_L|_{F_1}^k)$ as \begin{align*} \chi(F_1,\scH_L|_{F_1}^k) &= k^n\int_{F_1} \frac{c_1(\scH_L|_{F_1})^n}{n!} + O(k^{n-1}), \\ &= k^n\int_{\scY}\frac{c_1(\scH_L)^n.c_1(H_1)}{n!} + O(k^{n-1}).\end{align*} This gives $$\chi(\scY,\scH_L^k) = \chi(\scY,\scH_L^k\otimes H_1^{-1}) + k^n\int_{\scY}\frac{c_1(\scH_L)^n.c_1(H_1)}{n!} + O(k^{n-1}).$$

Continuing this process, the restriction exact sequence associated to $F_2\in |H_2|$ is $$0 \to \scH_L^k\otimes H_1^{-1} \to \scH_L^k\otimes H_1^{-1} \otimes H_2 \to \scH_L^k\otimes H_1^{-1}\otimes H_2|_{F_2}\to 0.$$ This similarly gives $$\chi(\scY,\scH_L^k\otimes H_1^{-1}) =  \chi(\scY,\scH_L^k\otimes H_1^{-1} \otimes H_2) - k^n\int_{\scY}\frac{c_1(\scH_L)^n.c_1(H_2)}{n!} + O(k^{n-1}).$$

Summing up, we have \begin{align*} 
\chi(\scY,\scH_L^k) &= \chi(\scY,\scH_L^k\otimes H_1^{-1}) + k^n\int_{\scY}\frac{c_1(\scH_L)^n.c_1(H_1)}{n!} + O(k^{n-1}), \\
&= \chi(\scY,\scH_L^k\otimes H_1^{-1} \otimes H_2) + k^n\int_{\scY}\frac{c_1(\scH_L)^n.(c_1(H_1\otimes H_2^{-1}))}{n!}+ O(k^{n-1}), \\
&= \chi(\scY,\scH_L^k\otimes\scH_T^{-1})+ \int_{\scY}\frac{c_1(\scH_L)^n.c_1(\scH_T)}{n!} + O(k^{n-1}), \end{align*} as required.

\item[(v)] Our scheme $X_0\hookrightarrow \pr(H^0(X_0,\scL_0^k\otimes\scT_0^{-1})^*)= \pr^{N_k}$ is embedded such that the $\C^*$-action on $X_0$ lifts to the overlying projective space and its tautological bundle. Thus there is another associated bundle $$(\scP,\scO_{\scP}(1)) = \scO_{\pr^1}(1)^* \times_{\C^*} (\pr^{N_k}, \scO_{\pr^{N_k}}(1)).$$ In this situation Donaldson \cite[Proposition 3]{D2} shows that there exists a K\"ahler form $\Omega$ representing $\scO_{\scP}(1)$ on $\scP$ satisfying $$\Omega = -h_B\tilde{\omega}_{FS} +  \omega_{FS},$$ where we have denoted by $\tilde{\omega}_{FS}$ the Fubini-Study metric of $\pr^1$ and $\omega_{FS}$ the Fubini-Study metric of $\pr^{N_k}$. Note that our sign convention is different to Donaldson's, which leads to the minus sign in the above formula.

Working over the smooth locus of $\scY$ and integrating over the fibres of the map $\eta: \scY\to\pr^1$ gives \begin{align*} 
\int_{\scY}\Omega^{n+1} &= -\int_{\pr^1}\tilde{\omega}_{FS}\int_{|\scX_0|}h_B \omega^n_{FS}, \\
 &=-\int_{|\scX_0|} h_B \omega^n_{FS}.\end{align*}

In terms of cohomology classes, we have $$\int_{\scY}\frac{\Omega^{n+1}}{(n+1)!} = \int_{\scY} \frac{c_1(\scH_L^{k}\otimes \scH_T^{-1})^{n+1}}{(n+1)!}. $$

Expanding the integral gives $$ -\int_{|X_0|} h_B \frac{\omega_{FS}^n}{n!} = k^{n+1}\int_\scY \frac{c_1(\scH_L)^{n+1}}{(n+1)!}-k^n\int_\scY \frac{c_1(\scH_L)^n.c_1(\scH_T)}{n!} + O(k^{n-1}).$$ Letting $\scZ$ be the subscheme of $\scY$ corresponding to $D_0$, we similarly get $$ -\int_{|D_0|} h_B \frac{\omega_{FS}^{n-1}}{(n-1)!} = k^n\int_{\scZ} \frac{c_1(\scH_L)^{n}}{n!} + O(k^{n-1}). $$ The result follows with $$c =-\int_{\scZ} \frac{c_1(\scH_L)^{n}}{n!} + \int_\scY \frac{c_1(\scH_L)^n.c_1(\scH_T)}{n!}.$$ 

Remark that as \emph{a priori} the smooth locus of $\scY$ has no $\C^*$-action, we cannot apply Donaldson's argument directly to $\scY$, which leads us to first work on the overlying projective space.

\item[(vi)] 

Given a twisted test configuration,  the $\C^*$-action on $\scX$ fixes the central fibre $(\scX_0,\scL_0,\scT_0)$. Denote by $B_k$ the infinitesimal generator of this action, with corresponding polynomials \begin{align*}
\chi(\scX_0, \scL_0^k\otimes\scT_0^{-1})&=a_{0,t}k^n+O(k^{n-1}), \\
\tr(B_k) &= b_{0,t}k^{n+1}+O(k^{n}),\\
\tr(B_k^2) &= d_{0,t}k^{n+2}+O(k^{n+1}).
\end{align*} We define the \emph{twisted} $L^2$\emph{-norm} of $(\scX,\scL,\scT)$ to be $$\|\scX\|_{2,t} = d_{0,t} - \frac{b^2_{0,t}}{a_{0,t}}.$$ 

With this definition in place, we prove the required statement in two steps. First we show that $$\|\underline{B}\| = \|\scX\|_{2,t}k^{n+2} + O(k^{n+1}),$$ then we show $\|\scX\|_{2,t} = \|\scX\|_{2}.$ For the first step, we use the same method as Donaldson \cite[Section 5]{D2}.  By direct calculation we have \begin{align*}\|\underline{B}\|^2 &= \tr(\underline{B}_k^2), \\
&= \tr ( \underline{B}_k B_k), \\
&= \tr\left(\left(B_k - \frac{\tr B_k}{h(k)}\right)B_k\right), \\
&= \tr(B_k^2) - \frac{\tr(B_k)^2}{h(k)}, \\
&= \left(d_{0,t} - \frac{b^2_{0,t}}{a_{0,t}}\right)k^{n+2} + O(k^{n+1}), \\
&= \|\scX\|_{2,t}k^{n+2} + O(k^{n+1}). \end{align*}

We wish to show $$d_{0,t} - \frac{b^2_{0,t}}{a_{0,t}} = d_{0} - \frac{b^2_{0}}{a_{0}}.$$ By parts $(iii),(iv)$ we have $b_{0,t}$ and $a_{0,t} = a_{0}$. Therefore it suffices to show $d_{0,t}= d_{0}.$

Now, by a similar calculation to part $(v)$, forming an associated bundle with base $\pr^2$ instead of $\pr^1$, we get that $$\tr(B_k^2) = \int_{|X_0|} h_B^2\frac{\omega_{FS}^n}{n!}+ O(k^{n+1}).$$ Again as in part $(v)$, from working on the total space of the associated bundle we see that $$\int_{|X_0|} h_B^2\frac{\omega_{FS}^n}{n!} = \int_{|X_0|} h_A^2\frac{\hat{\omega}_{FS}^n}{n!} + O(k^{n+1}),$$ where we denote by $\hat{\omega}_{FS}$ the Fubini-Study metric arising from an equivariant embedding of the test configuration by sections of $L^k$, with corresponding Hamiltonian $h_A$. This gives $d_{0,t}= d_{0},$ as required. \end{itemize} \end{proof}

We now prove that the twisted Donaldson-Futaki invariant is linear in $T$, as mentioned in Remark \ref{entryremark}. We first briefly recall the setup. When $T$ is semi-positive, one first writes write $T=H_1\otimes H_2^{-1}$ for very ample line bundles $H_1,H_2$. A twisted test configuration for $(X,L,T)$ is a test configuration $(\scX,\scL)$ whose $\C^*$-action lifts equivariantly to $\scH_1$ and $\scH_2$ where $\scH_1$ and $\scH_2$ restrict to $H_1,H_2$ respectively on the non-zero fibres over $\C$. 

We have a $\C^*$-action on $H^0(\scX_0,\scL_0^k)$ for all $k\gg 0$, and just as in Definition \ref{twisteddf} we have a corresponding leading weight coefficient $b_0$. Similarly, taking the flat limit $D_0$ of $D\in |H_1|$ we have a $\C^*$-action on $H^0(D_0,\scL_0|_{D_0}^k)$. We define $\hat{b}_0$ to be the leading weight coefficient for a general $D\in |H_1|$, in the sense of Definition \ref{currentsasaverage}. Denote $$J_{H_1}(\scX,\scL) = \frac{\hat{b}_0 a_0 - b_0 \hat{a}_0}{a_0},$$ and similarly $J_{H_2}(\scX,\scL)$ using $\hat{b}'_0$ for a general $D'\in |H_2|$. Then the \emph{twisted Donaldson-Futaki invariant} of $(\scX,\scL,\scT)$ is $$\DF(\scX,\scL,\scT) = \frac{b_0a_1 - b_1a_0}{a_0} + J_{H_1}(\scX,\scL) - J_{H_2}(\scX,\scL).$$

\begin{lemma}\label{linearitylemma} The twisted Donaldson-Futaki invariant is independent of choice of $H_1,H_2$. Moreover, for arbitrary very ample line bundles $T_1,T_2$ we have $$J_{T_1+T_2}(\scX,\scL)=J_{T_1}(\scX,\scL)+J_{T_2}(\scX,\scL),$$ i.e. the twisted part of the twisted Donaldson-Futaki invariant is linear in the twisting. \end{lemma}

\begin{proof} We first show that $\DF(\scX,\scL,\scT)$ is independent of choice of $H_1,H_2$. By asymptotic Riemann-Roch we have $\hat{a}_0 = \frac{T.L^{n-1}}{(n-1)!}$, which is clearly independent of $H_1,H_2$. 

Embed the twisted test configuration $(\scX,\scL,\scT)$ into projective space $$\pr(H^0(\C,\pi_*(\scL_0^k))^*)\times \C\cong \pr(H^0(X,L^k)^*)\times \C$$ using Proposition \ref{equivariantembedding}. Let $h_A$ be the corresponding Hamiltonian as in Section \ref{fixedproj}. By \cite[Equation (38)]{LS}, we have $$\hat{b}_0 = -\int_{|D_0|}h_A \frac{\omega_{FS}}{(n-1)!},$$ and similarly $$\hat{b}'_0 = -\int_{|D'_0|}h_A \frac{\omega_{FS}}{(n-1)!},$$ where as above $D,D'$ are chosen to be general.

Choose an arbitrary $\frac{1}{2}\alpha \in c_1(T)$, with $\alpha = \eta_1 - \eta_2$ and $\frac{1}{2}\eta_i \in c_1(H_i)$ positive $(1,1)$-forms. Note such a choice of $\eta_1, \eta_2$ exists by the proof of Lemma \ref{sposcurrents}. Lemma \ref{currentsasaverage} implies \begin{align*}\lim_{t\to 0} \frac{1}{2}\int_{\varphi_k^t(X)} h_A \frac{\alpha\wedge\omega^{n-1}_{FS}}{(n-1)!} &=\lim_{t\to 0} \frac{1}{2}\int_{\varphi_k^t(X)} h_A \frac{\eta_1 - \eta_2\wedge\omega^{n-1}_{FS}}{(n-1)!}, \\ &= \int_{|D_0|}h_A \frac{\omega^{n-1}_{FS}}{{(n-1)!}}-\int_{|D'|}h_A \frac{\omega^{n-1}_{FS}}{{(n-1)!}}, \\ &= -\hat{b}_0 + \hat{b}'_0.\end{align*} The integral $$\frac{1}{2}\int_{\varphi_k^t(X)} h_A \frac{\alpha\wedge\omega^{n-1}_{FS}}{(n-1)!}$$ does not depend on $H_1,H_2$, therefore the weight coefficient $\hat{b}_0 - \hat{b}'_0$, and hence the Donaldson-Futaki invariant $\DF(\scX,\scL,\scT)$, are also independent of choice of $H_1,H_2$.

We now show that the twisted part of the twisted Donaldson-Futaki invariant is linear in the twisting, using the same method. Again it is clear that the $\hat{a}_0$ term is additive in $T_1,T_2$. Therefore it suffices to show that the $\hat{b}_0$ is additive. Embedding the twisted test configuration $(\scX,\scL,\scT)$ into projective space $\pr(H^0(\C,\pi_*(\scL_0^k))^*)\times \C$ as above, the $\hat{b}_0$ term is given as $$\hat{b}_0 = -\lim_{t\to 0} \frac{1}{2}\int_{\varphi_k^t(X)} h_A \frac{\alpha\wedge\omega^{n-1}_{FS}}{(n-1)!},$$ for arbitrary positive $\alpha \in c_1(T_i)$. This integral is clearly additive in $T_1,T_2$, hence so is the $\hat{b}_0$ term.\end{proof}

\subsection{Perturbation argument}

We now use a perturbation argument to show that the existence of a twisted cscK metric implies uniform stability when the twisting $T$ is ample. 

\begin{lemma}\label{perturbationlemma} Let $\frac{1}{2}\alpha \in c_1(T)$ be positive. Given a solution to $$S(\omega) - \Lambda_{\omega}\alpha = C_{\alpha},$$ there exists a solution to \begin{equation}\label{analpert}S(\omega) - \Lambda_{\omega}(\alpha-\epsilon\omega_0) = C_{\alpha-\epsilon\omega_0}\end{equation} for some $\omega_0 \in [\omega]$ and some $\epsilon>0$. In particular, $(X,L,T-\epsilon L)$ is twisted K-semistable.\end{lemma}

\begin{proof} One can take $\omega = \omega_0$. In this case one has $\Lambda_\omega \omega = n$, so $\omega$ itself is a solution to equation (\ref{analpert}). One would also expect to be able to prove openness in $\alpha$ using an implicit function theorem argument similar to LeBrun-Simanca \cite[Theorem 4]{LeSi}. Twisted K-semistability of $(X,L,T-\epsilon L)$ then follows from Proposition \ref{calabifunctional}, where we choose $\epsilon>0$ such that $\alpha-\epsilon\omega$ is positive. \end{proof}

We can now prove Theorem \ref{stabilityofcscK} $(i)$.

\begin{proof}[Proof of Theorem  \ref{stabilityofcscK} (i)]  Take an arbitrary twisted test configuration $(\scX,\scL,\scT)$. We show that $$\DF(\scX,\scL,\scT)\geq\epsilon \|\scX\|_m.$$ Note that by Lemma \ref{linearitylemma} the Donaldson-Futaki invariant of a twisted test configuration is additive in the twisting. Explicitly, take a general section $D\in |L|$, with corresponding Hilbert and weight polynomials for the central fibre of the induced test configuration \begin{align*}
&\tilde{h}(k)=\tilde{a}_0k^{n-1}+O(k^{n-2}), \\
&\tilde{w}(k) = \tilde{b}_0k^{n}+O(k^{n-1}).\end{align*} If we define $$J_L(\scX,\scL) = \frac{\tilde{b}_0a_0 - b_0\tilde{a}_0}{a_0}$$ in the notation of Definition \ref{twisteddf}, then for all twisted test configurations we have \begin{equation}\label{linearity} \DF(\scX,\scL,\scT) = \DF(\scX,\scL,\scT-\epsilon \scL)+ \epsilon J_L(\scX,\scL).\end{equation} 

Since by Lemma \ref{perturbationlemma} $(X,L,T-\epsilon L)$ is K-semistable, we have $\DF(\scX,\scL,\scT-\epsilon \scL)\geq 0$ and so equation (\ref{linearity}) implies $$\DF(\scX,\scL,\scT) \geq \epsilon J_L(\scX,\scL).$$

Since in this case $D\in |L|$, we have $na_0 = \tilde{a}_0$. Recall from Definition \ref{minnorm} that the minimum norm of a test configuration is given by $$\|\scX\|_m =\sum_j (b_{0,j} - \lambda_j a_{0,j}),$$ where $b_{0,j}, a_{0,j}$ are the corresponding terms on each irreducible component of the central fibre and $\lambda_j$ is the minimum weight of the induced $\C^*$-action on the reduced support of the same component. Note also that $b_0 = \sum_j b_{0,j}$, as one can see explicitly in the integral form for the minimum norm in equation (\ref{highestweight}). In this case, Lejmi-Sz\'ekelyhidi \cite[Theorem 12]{LS} show that $$\tilde{b}_{0,j} = (n+1)b_{0,j} - \lambda_j a_{0,j}.$$

Finally, we have $$J_L(\scX,\scL) = \sum_j ((n+1)b_{0,j} -\lambda_j a_{0,j}) - nb_0  = \|\scX\|_m,$$ as required.
\end{proof}

\begin{remark}\label{realtwisting} The above perturbation method implies that Theorem \ref{stabilityofcscK} holds also when $T=\beta L$, for $\beta$ a \emph{real} number, for example along the Aubin continuity method. Suppose $(X,L,T)$ admits a twisted cscK metric in this case. By Lemma \ref{perturbationlemma}, $(X,L,T-\delta L)$ also admits a cscK metric for all $\delta$ sufficiently small. Take $\delta>0$ such that $T - \delta L$ is an ample $\Q$-line bundle. Note that we are using here that $T=\beta L$, such a $\delta$ may not exist for more general $T$. Then for all twisted test configurations $(\scX,\scL,\scT)$ we have $$\DF(\scX,\scL,\scT-\delta \scL) \geq \epsilon \|\scX\|_m,$$ by Theorem \ref{stabilityofcscK}. Using the notation as above, we have $J_L(\scX,\scL) = \|\scX\|_m$. Therefore $$\DF(\scX,\scL,\scT) \geq (\epsilon+\delta) \|\scX\|_m,$$ showing uniform twisted K-stability. \end{remark}

\begin{remark}\label{equivalentminnorm} The proof of Theorem  \ref{stabilityofcscK} (i) above also gives another interpretation of the minimum norm, namely \begin{align*} \|\scX\|_m &= J_L(\scX,\scL), \\ &=\frac{\tilde{b}_0a_0 - b_0\tilde{a}_0}{a_0},\end{align*} using the notation of the proof. \end{remark} 

\section{Sufficient geometric criteria for uniform and twisted K-stability}

\subsection{A blowing-up formalism for twisted K-stability}

One way of constructing twisted test configurations is by blowing-up flag ideals. We show in this section that these are also sufficient to check twisted K-stability. We allow $X$ to be a singular projective variety, provided $X$ is normal and $\Q$-Gorenstein, so that $K_X$ exists as a $\Q$-Cartier divisor.

\begin{definition} A flag ideal on $X$ is a coherent ideal sheaf $\scI$ on $X\times \C$ of the form $\scI = I_0 + (t)I_1 +\hdots +(t^N)$ with $I_0\subseteq I_1 \subseteq \hdots \subseteq I_{N-1}\subseteq \scO_X$ a sequence of coherent ideal sheaves. The ideal sheaves $I_j$ therefore correspond to subschemes $Z_0\supseteq Z_1\supseteq\hdots\supseteq Z_{N-1}$ of $X$. \end{definition}

Blowing-up $\scI$ on $X\times\C$ gives a map $$ \pi: \tilde{\scB}=Bl_{\scI}(X\times\C)\to X\times\C.$$ Denote by $E$ the exceptional divisor of $\pi$, so that $\scO(-E)=\pi^{-1}\scI$. Denote by $\scL, \scT$ the pullbacks of $L$ and $T$ from $X$ to $\tilde{\scB}$. By \cite[Remark 5.2]{RT}, the induced map $\tilde{\scB}\to \C$ is flat. The natural $\C^*$-action on $X\times\C$, acting trivially on $X$, lifts to an action on $(\tilde{\scB},\scL,\scT)$. 

\begin{proposition}\label{resofindet}\cite[Corollary 3.11]{O3} Assume that $(X,L,T)$ is a normal $\Q$-Gorenstein polarised variety and let $(\scX,\scH_L,\scH_T)$ be a twisted test configuration with non-zero fibre $\scH_{L,t} \cong L^r$ for some $r>0$, and with $\scX$ normal. Then there exists a flag ideal $\scI$ and a $\C^*$-equivariant map $$\pi: (\tilde{\scB},\scL^r-E) \to (\scX,\scH_L)$$ with $\tilde{\scB}$ normal and such that $\scL^r-E=\pi^*\scH_L$ is relatively semi-ample over $\C$.  \end{proposition}

The $\C^*$-action on $(\tilde{\scB},\scL-E)$ fixes the central fibre, hence there is a  corresponding weight polynomial $\wt H^0 (\tilde{\scB}_0,(\scL^k-kE)_0).$ Taking any $D\in |T|$, note that the blow-up of $D\times\C$ along the restriction of $\scI$ equals the proper transform of $D\times\C$ under $\pi$, denote this by $\tilde{\scB}_D$. Again by \cite[Remark 5.2]{RT}, $\tilde{\scB}_D\to\pr^1$ is flat, so there are corresponding weight and Hilbert polynomials on the central fibre. Since the map $\pi: \tilde{\scB}\to \scX$ is $\C^*$-equivariant, the weight polynomials, and therefore the twisted Donaldson-Futaki invariants, of $(\tilde{\scB},\scL-E,\scT)$ and $(\scX,\scL,\scT)$ are equal \cite[Proposition 5.1]{RT}. For the same reason, the norms satisfy $\|\scB\|_m = \|\scX\|_m$. 

We show in Theorem \ref{trivialitytheorem} that normalisation preserves the minimum norm, therefore there is no loss of generality in the above Proposition in assuming normality of $\scX$. That is, if $\scX$ is not normal, we take its normalisation, which by \cite[Proposition 5.1]{RT} has \emph{lower} Donaldson-Futaki invariant, and the same minimum norm. Therefore to check uniform twisted K-stability, it is enough to check for test configurations with normal total space.

\begin{corollary} To show uniform twisted K-stability, it is sufficient to show $$\DF(\tilde{\scB},\scL^r-E,\scT)>\epsilon \|\scB\|_m$$ for all $r>0$ and for all flag ideals $\scI\neq(t^N)$ with $\tilde{\scB}$ normal and with $\scL^r-E$ relatively semi-ample over $\C$. \end{corollary}

Note that if $\scI = (t^N)$, then the blow-up results in a trivial test configuration. We assume for notational convenience throughout that $r=1$, as this will make no difference to our results. The benefit of this formalism is that, for test configurations arising from flag ideals, we can give an explicit intersection-theoretic formula for the Donaldson-Futaki invariant. In order to apply the machinery of intersection theory, we compactify $\tilde{\scB}$ as follows. The flag ideal $\scI$ naturally induces a coherent ideal sheaf on $X\times\pr^1$; blowing this up gives another scheme $\scB$ with line bundles which we also denote by $\scL,\scT,E$, abusing notation. Similarly, we compactify $\tilde{\scB}_D$ to a corresponding scheme  $\scB_D$. The corresponding Donaldson-Futaki invariants of $\tilde{\scB}$ and $\scB$ are equal, so from now on we work with the compactified space $\scB$. To ease notation, we introduce the \emph{twisted slope} of $(X,L,T)$.

\begin{definition}We define the \emph{twisted slope} of $(X,L,T)$ to be $$\mu(X,L,T)=\frac{(-K_X-2T).L^{n-1}}{L^n}=\frac{\int_X(c_1(X)-2c_1(T)).c_1(L)^{n-1}}{\int_Xc_1(L)^n}.$$ \end{definition}

\begin{remark} The sign of the twisted slope is governed by the geometry of $K_X+2T$. When $K_X+2T$ is ample, the slope is positive and we call this the twisted general type case. The twisted Calabi-Yau case is when $K_X+2T$ is numerically trivial, and finally the twisted Fano case is when $-K_X-2T$ is ample.\end{remark}

\begin{proposition}\label{dfformulapf}The twisted Donaldson-Futaki invariant of a blow-up along a flag ideal $\scI$ as above is given up to multiplication a positive constant depending only on the dimension $n$ of $X$ \begin{equation}\label{dfform} \DF(\scB,\scL-E,\scT) = \frac{n}{n+1}\mu(X,L,T)(\scL-E)^{n+1} +(\scL-E)^n.(\scK_X + 2\scT + K_{\scB/X\times\pr^1}).\end{equation}Here we have denoted by $\scK_X$ the pull back of $K_X$ to $\scB$. The term $K_{\scB/X\times\pr^1}$ is the relative canonical class. Since $\scB$ is normal, its canonical class $K_{\scB}$ exists as a Weil divisor and the intersection numbers in the above formula are well defined. The twisted slope $\mu(X,L,T)$ is computed on $X$, while the remaining intersection numbers are computed on $\scB$.  The formula for the twisted Donaldson-Futaki invariant of a twisted test configuration of the form $(\scB, \scL^r-E,\scT)$ can be obtained by replacing $\scL$ with $\scL^r$ in equation (\ref{dfform}).
\end{proposition}

\begin{proof} We follow the method of Odaka \cite[Theorem 3.2]{O2} in the untwisted case, and assume $r=1$ for notational simplicity. Recall that the twisted Donaldson-Futaki invariant is defined as $$\DF(\scX,\scH_L,\scH_L) = \frac{b_0a_1 - b_1a_0}{a_0} + \frac{\hat{b}_0 a_0 - b_0 \hat{a}_0}{a_0},$$ with notation as in Definition \ref{twisteddf}.

By \cite[Theorem 3.2]{O2}, we have \begin{equation}\label{wtformula}\wt(H^0(\scB_0,(\scL^k-kE)|_0)) = \chi(\scB,\scL^k-kE) - \chi(X\times\pr^1,\scL^k) + O(k^{n-1}),\end{equation} where by abuse of notation we have denoted by $\scL$ the pullback of $L$ to $X\times\pr^1$. The corresponding formula for the limit of a divisor $D$ on the central fibre is $$\wt(H^0(\scD_0,(\scL^k-kE)|_0)) = \chi(\scD,\scL^k-kE) +O(k^{n-1}).$$  

Using asymptotic Riemann-Roch for normal varieties \cite[Lemma 3.5]{O3}, from formula (\ref{wtformula}) we see that $$b_0 = \frac{(\scL-E)^{n+1}}{(n+1)!}.$$ The order $k^n$ term in the weight polynomial is given by $$(2n!)b_1 = -(\scL-E)^n.(K_{\scB}) + (\scL)^{n}.(K_{X\times\pr^1}).$$ Since the canonical class of a product is the sum of the canonical classes,  denoting by $\scK_{\pr^1}$ the pullback of $K_{\pr^1}$ to $\scB$ and $X\times\pr^1$ we have $$(\scL)^{n}.(K_{X\times\pr^1}) = (\scL)^{n}.\scK_{\pr^1} = -2L^n.$$ On the other hand, we have $$(\scL-E)^n.\scK_{\pr^1} = -2L^n,$$ since the map $\pi: \scB \to \pr^1$ is flat and $(\scL-E)\cong L$ over non-zero fibres of $\pi$. Therefore $$(\scL)^{n}.(K_{X\times\pr^1}) = (\scL-E)^n.\scK_{\pr^1}$$ and we have \begin{align*}(2n!)b_1 &= -(\scL-E)^n.(K_{\scB} - \scK_{\pr^1}), \\ &= -(\scL-E)^n.(\scK_X + K_{\scB/X\times\pr^1}).\end{align*}

Similarly for a fixed $D$ the relevant weight term for the proper transform $\scB_D$ is  $$\hat{b}_0 = (\scL-E)^n|_{\scB_D}.$$ We first consider the case when $D\in |2T|$ with $2T$ very ample. 

By definition of the twisted slope, we have $\frac{n}{2}\mu(X,L,T) = \frac{(a_1-\hat{a}_0)}{a_0}.$ The twisted Donaldson-Futaki invariant is therefore given as $$\DF(\scX,\scL,\scT) = \frac{n}{2}\mu(X,L,T)b_0 - (b_1-\hat{b}_0).$$

Substituting in gives, up to multiplication by the positive dimensional constant $2n!$ \begin{equation}\label{bigdf} \DF(\scB,\scL-E,\scT)  = \frac{n}{n+1}\mu(X,L,T)(\scL-E)^{n+1} +(\scL-E)^n.(\scK_X + 2\scT + K_{\scB/(X,D)\times\pr^1,exc}).\end{equation} Here we have denoted $$ K_{\scB/(X,D)\times\pr^1,exc} = K_{\scB/X\times\pr^1} -  \pi^*(D\times\pr^1) + \pi^{-1}_*(D\times\pr^1).$$ This is the exceptional part of the discrepancy of $(X\times\pr^1,D\times\pr^1)$ under the birational map $\pi$ \cite[Definition 3.3]{JaKo}. For a general $D$, this then becomes the discrepancy of the \emph{linear system} $|2T|$ \cite[Definition 4.6]{JaKo}. If $|2T|$ is basepoint free, by \cite[Lemma 4.7]{JaKo} have that for general $D$ $$K_{\scB/(X,D)\times\pr^1,exc} = K_{\scB/X\times\pr^1}.$$ Indeed  $\pi^*(D\times\pr^1) - \pi^{-1}_*(D\times\pr^1)$ is trivial if and only if $\scI$ has no component contained in $D\times\pr^1$, since $|2T|$ is basepoint free this is the case for general $D$. When $|2T|$ is not basepoint free, by definition of the Donladson-Futaki invariant one writes $T=H_1\otimes H_2^{-1}$, with $H_1,H_2$ very ample, and takes general sections $D_1\in |2H_1|, D_2\in |2H_2|$. The required formula then follows as above.

\end{proof}

\begin{remark}Proposition \ref{dfformulapf} implies that one can define twisted K-stability when $L,T\in \Amp_{\R}(X)$ are ample $\R$-line bundles. It would be interesting to extend Theorem \ref{stabilityofcscK} to this setting. \end{remark}

Recall from Definition \ref{twisteddf} that the definition of a twisted test configuration requires a scheme $\pi:\scX\to\C$ together with $\C^*$-equivariant line bundles $\scL,\scT$ which restrict to $L$ and $T$ respectively over the non-zero fibres. On the other hand, the definition of the twisted Donaldson-Futaki invariant does not require any lifting of the $\C^*$-action on $\scX$ to any line bundle $\scT$ restricting to $T$ over the non-zero fibres. In practice, it is difficult to check if a $\C^*$-action lifts to $\scT$ for some choice of $\scT$. Recall that in Definition \ref{untwistedkstab} we defined a \emph{test configuration} to be a scheme $\pi:\scX\to\C$ with $\C^*$-action lifting to $\scL$. We can therefore define the twisted Donaldson-Futaki invariant of an arbitrary test configuration, even if the $\C^*$-action does not lift to $\scT$ for any choice of $\scT$. In the following Proposition we show that one can approximate any test configuration by twisted test configurations, i.e. such that the action lifts to $\scT$ for some $\scT$, with arbitrarily close twisted Donaldson-Futaki invariant.

\begin{proposition}\label{approximation} One can approximate any test configuration for a normal variety $(X,L,T)$, whose $\C^*$-action may or may not lift to $\scT$ for some $\scT$, by twisted test configurations with arbitrarily close twisted Donaldson-Futaki invariant. In particular, to check twisted K-semistability, one can assume the test configuration embeds only through sections of $L^k$. \end{proposition}

\begin{proof} Let $(\scX,\scL)$ be an arbitrary test configuration whose action \emph{may or may not} lift to $\scT$. By Proposition \ref{resofindet}, there exists a flag ideal $\scI$ with blow-up $(\scB,\scL-E)$ with the same twisted Donaldson-Futaki invariant as $(\scX,\scL)$. Note that the $\C^*$-action on $\scB$ certainly lifts to $\scT$. Therefore when $\scL-E$ is relatively ample, rather than just relatively semi-ample, we have produced a test configuration with the \emph{same} Donaldson-Futaki invariant. In general, one notes that the definition of the Donaldson-Futaki invariant of a blow-up $(\scB,\scL- (1-\epsilon)E)$ is continuous when varying $\epsilon$. When $\epsilon=0$ this is the Donaldson-Futaki invariant of $(\scX,\scL)$, and when $\epsilon>0$ this gives a genuine twisted test configuration. By continuity the result follows. \end{proof}

\begin{remark} Proposition \ref{resofindet} shows that if $DF(\scB,\scL-E,\scT)\geq\epsilon\|\scB\|_m$ for all flag ideals $\scI$ with $\scL-E$ relatively semi-ample, then $(X,L,T)$ is uniformly twisted K-stable. It is not clear whether the reverse is true, namely if uniform twisted K-stability implies $DF(\scB,\scL-E,\scT)>\epsilon\|\scB\|_m$ for all such flag ideals. The difference is in the strictly semi-ample case. Suppose $\scI$ is a flag ideal with $DF(\scB,\scL-E,\scT)=0$, with $\scL-E$ semi-ample but not ample. Contracting gives a genuine test configuration, however the action may not lift to $\scT$ for any such $\scT$. Hence even if $(X,L,T)$ is uniformly twisted K-stable, one cannot conclude that $DF(\scB,\scL-E,\scT)\geq\epsilon\|\scB\|_m$ for such $\scI$. However when $L=T$ it is clear that the two are equivalent. \end{remark} 

At several points we will need the following positivity results for intersections on $\scB$.

\begin{lemma}\label{inequalities}\cite[Proposition 4.3, Theorem 2.6]{OS}\cite[Equation (3)]{O2}\cite[Lemma 3.7]{RD} Let $\scB$ be the blow-up of $X\times\pr^1$ along a flag ideal $\scI\neq (t^N)$. Let $L$ and $R$ be ample and nef divisors respectively on $X$, with $p^*L = \scL$ and $p^*R = \scR$ where $p: \scB\to X$ is the natural morphism given by the composition of the blow-up map and the projection. Denote by $E$ the exceptional divisor of the blow-up, and assume $\scL-E$ is relatively semi-ample. Then the following intersection theoretic inequalities hold. \begin{itemize} \item[(i)] $(\scL-E)^n.\scR \leq 0$, \item[(ii)] $(\scL-E)^n.E > 0$,  \item[(iii)] $(\scL-E)^n.(\scL+nE)>0$. \end{itemize}\end{lemma}

\begin{remark}\label{explicitnorm}We can also give an intersection-theoretic definition of the minimum norm. Indeed, using also Remark \ref{equivalentminnorm} and following the proof of Proposition \ref{dfformulapf} we have \begin{align*}  \|\scX\|_m &= J_L(\scX,\scL), \\ &=(\scL-E)^n\left(\frac{1}{n+1}(\scL+nE)\right).\end{align*} By Lemma \ref{inequalities} $(iii)$, the minimum norm of a blow-up along a flag ideal is therefore strictly positive provided $\scI\neq (t^N)$. In Section $4$ we make further remarks on the various notions of triviality of test configurations. \end{remark}

Our various sufficient conditions for twisted K-stability will exploit the mildness of the singularities of $X$, as such we require the following definition.

\begin{definition}\label{lc} Let $X$ be a normal variety and let $D=\sum d_iD_i$ be a divisor on $X$ such that $K_X+D$ is $\Q$-Cartier, where $D_i$ are prime divisors. Let $\pi: Y\to X$ be an arbitrary birational map with $Y$ normal. We can then write \begin{equation*}K_Y - \pi^*(K_X+D) \equiv \sum a(E_i,(X,D))E_i\end{equation*} where $E_i$ is either an exceptional divisor or $E_i=\pi_*^{-1}D_i$ for some $i$, so that either $E_i$ is exceptional or the proper transform of a component of $D$. We often abbreviate $a(E_i,(X,D))$ to $a_i$. We say the pair $(X,D)$ is \begin{itemize}
\item \emph{log canonical} if $a(E_i,(X,D))\geq -1$ for all $E_i$,
\item \emph{Kawamata log terminal} if $a(E_i,(X,D)) > -1$ for all $E_i$,
\item \emph{purely log terminal} if $a(E_i,(X,D)) > -1$ for all \emph{exceptional} $E_i$,\end{itemize}
with $Y$ normal.
 By \cite[Lemma 3.13]{JK} it suffices to check this property for $Y\to X$ a log resolution of singularities.\end{definition}

\begin{definition} We say a variety $X$ is \emph{log canonical} or \emph{Kawamata log terminal} respectively if $(X,0)$ is log canonical or Kawamata log terminal. Note that log canonical varieties are normal by assumption. \end{definition}

\subsection{Alpha invariants and the continuity method}

In this section we focus on the situation when $(X,L,T)$ is twisted Fano, i.e. when $-K_X-2T$ is ample and the slope is \emph{positive}. The condition we prove relies on the \emph{alpha invariant}, which we now define.

\begin{definition} Let $X$ be a normal variety, and let $D$ be a $\Q$-Cartier divisor. The \emph{log canonical threshold} of $(X,D)$ is \begin{equation*} \lct(X,D) = \sup \{\lambda\in\Q_{>0}\ | \  (X,\lambda D) \mathrm{\ is \ log \ canonical} \}. \end{equation*}We define the \emph{alpha invariant} of $(X,L)$ to be \begin{equation*} \alpha(X,L) = \inf_{m \in \Z_{>0}}\inf_{D \in |mL|} \lct(X,\frac{1}{m}D).\end{equation*} \end{definition}

\begin{theorem}\label{alpha} Let $(X,L,T)$ be a $\Q$-Gorenstein Kawamata log terminal variety $X$ with canonical divisor $K_X$. Suppose that 
\begin{itemize}
\item[(i)]  $\alpha(X,L)>\frac{n}{n+1}\mu(X,L,T)$ and
\item[(ii)]  $-(K_X +2T)\geq \frac{n}{n+1}\mu(X,L,T) L$.
\end{itemize}
Then $(X,L,T)$ is uniformly twisted K-stable with respect to the minimum norm.\end{theorem}

Note that for $c>0$ the alpha invariant satisfies the scaling property \begin{equation*}\alpha(X,cL) =\frac{1}{c} \alpha(X,L).\end{equation*} In particular, both conditions are independent of scaling $L$. 

To prove Theorem \ref{alpha} we use the following upper bound for the alpha invariant.

\begin{proposition}\cite[Proposition 3.6]{RD}\label{alphaprop}\cite[Proposition 3.1]{OS} Let $\scB$ be the blow-up of $X\times\pr^1$ along a flag ideal, with $\scB$ normal and $\scL-E$ relatively semi-ample over $\pr^1$. Denote the natural map arising from the composing the blow-up map and the projection map by $\Pi: \scB\to\pr^1$, and denote also

\begin{itemize}
\item the discrepancies as: $K_{\mathcal{B} / X \times \pr^1} = \sum a_iE_i$,
\item the multiplicities of $X\times \{0\}$ as: $\Pi^*(X\times \{0\}) = \Pi_*^{-1}(X\times \{0\}) + \sum b_i E_i$,
\item the exceptional divisor as: $\Pi^{-1}\mathcal{I} = \mathcal{O}_{\mathcal{B}}(-\sum c_i E_i)= \mathcal{O}_{\mathcal{B}}(-E).$\end{itemize} 

Then \begin{equation*} \alpha(X,L) \leq \min_i\left\{\frac{a_i-b_i+1}{c_i}\right\}. \end{equation*}\end{proposition}

We also require the following estimates for intersections with exceptional divisors. 

\begin{lemma}\label{discreptouniform} Let $\scB$ be a blow-up of $X\times\pr^1$ along a flag ideal, with all notation as in Proposition \ref{dfformulapf}. Suppose $X$ is Kawamata log terminal. Then \begin{itemize} \item[(i)]$(\scL-E)^n.(K_{\scB/X\times\pr^1} - \alpha(X,L) E)\geq 0,$ \item[(ii)] $(\scL-E)^n.\alpha(X,L)E> \epsilon \|\scB\|_m$, \item[(iii)] $(\scL-E)^n.K_{\scB/X\times\pr^1}> \epsilon \|\scB\|_m$, \end{itemize} for some $\epsilon>0$ independent of $(\scB,\scL-E)$. \end{lemma}

\begin{proof} \item[(i)]By Lemma \ref{inequalities} $(ii)$, it suffices to show that the exceptional divisor $$K_{\scB/X\times\pr^1} -  \alpha(X,L)E$$ is effective. For divisors $H_1,H_2$, we temporarily write $H_1 \geq H_2$ to mean the difference $H_1-H_2$ is effective. Then, using Proposition \ref{alphaprop}, we have \begin{align*} K_{\scB/X\times\pr^1} -  \alpha(X,L)E &\geq \sum a_i E_i -  \min_i\left\{\frac{a_i-b_i+1}{c_i}\right\}\sum c_i E_i, \\ &= \sum\left(\frac{a_i-b_i+1}{c_i} - \min_i\left\{\frac{a_i-b_i+1}{c_i}\right\} + \frac{b_i-1}{c_i}\right)c_i E_i, \\ & \geq 0.  \end{align*}

\item[(ii)]By Remark \ref{explicitnorm}, the minimum norm $\|\scB\|_m$ is given as \begin{align*}\|\scB\|_m &=(\scL-E)^n.\left(\frac{1}{n+1}(\scL+nE)\right), \\ &\geq \frac{n}{n+1}(\scL-E)^n.E,\end{align*} where the last inequality follows by Lemma \ref{inequalities} $(i)$. The alpha invariant $\alpha(X,L)$ is strictly positive by \cite[Proposition 1.4]{BBEGZ}. Setting $\epsilon = \alpha(X,L)$ proves that $$(\scL-E)^n.\alpha(X,L)E> \epsilon \|\scB\|_m,$$ as required.

\item[(iii)] This follows immediately from parts $(i),(ii)$.
  \end{proof}

\begin{proof}[Proof of Theorem \ref{alpha}] We show that $\DF(\scB,\scL^r-E,\scT)>\epsilon \|\scB\|_m$ for all flag ideals with $\scL^r-E$ relatively semi-ample over $\pr^1$. We assume $r=1$ for notational simplicity. We first split the Donaldson-Futaki invariant as \begin{align*} \DF(\scB,\scL-E,\scT) &= \frac{n}{n+1}\mu(X,L,T)(\scL-E)^{n+1} +(\scL-E)^n.(\scK_X + 2\scT + K_{\scB/X\times\pr^1}), \\ &= (\scL-E)^n.((\frac{n}{n+1}\mu(X,L,T)\scL + \\    &\ \ \ \ \ \ \ \ \ \ \ \ + \scK_X+2\scT) + (K_{\scB/X\times\pr^1} - \frac{n}{n+1}\mu(X,L,T)E)).\end{align*} By the second hypothesis of Theorem \ref{alpha}, we have $$-(K_X +2T)\geq \frac{n}{n+1}\mu(X,L,T) L,$$ by Lemma \ref{inequalities} $(i)$ this implies $$(\scL-E)^n.\left(\frac{n}{n+1}\mu(X,L,T)\scL + \scK_X+2\scT\right)\geq 0.$$ 

To control the second term we use Lemma \ref{discreptouniform}. Our assumption is $(1-\delta)\alpha(X,L) = \frac{n}{n+1}\mu(X,L,T)$ for some $\delta > 0$, so \begin{align*} (\scL-E)^n.&\left(K_{\scB/X\times\pr^1} - \frac{n}{n+1}\mu(X,L,T)E\right)  = \\ & \ \ \ \ \ \ \ \ \ \ \ \ \ \ \ \ \ \ \  = (\scL-E)^n.(K_{\scB/X\times\pr^1}-(1-\delta)\alpha(X,L)E).\end{align*}  Using Lemma \ref{discreptouniform} $(i),(ii)$ we see \begin{align*}(\scL-E)^n.\left(K_{\scB/X\times\pr^1} - \frac{n}{n+1}\mu(X,L,T)E\right) & \geq \delta (\scL-E)^n.\alpha(X,L)E, \\ &> \epsilon \|\scB\|_m,\end{align*} as required.
\end{proof}

\begin{remark} When $T=\scO_X$ and $L=-K_X$, Odaka-Sano \cite[Theorem 1.4]{OS} proved that $\alpha(X,-K_X)>\frac{n}{n+1}$ implies that $(X,-K_X)$ is K-stable, which is the algebraic counterpart to a famous theorem of Tian \cite[Theorem 2.1]{GT2}. Again when $T=\scO_X$, but with $L$ arbitrary, K-stability was proven by the author \cite[Theorem 1.1]{RD}. The uniformity statement in Theorem \ref{alpha} strengthens these results. 

When $T=\scO_X$, the criteria of Theorem \ref{alpha} imply coercivity of the Mabuchi functional, which is another condition conjecturally equivalent to the existence of a cscK metric \cite{RD2}. Again for $T=\scO_X$, a recent result due to Li-Shi-Yao \cite[Theorem 1.1]{LSY} also implies coercivity of the Mabuchi functional under criteria which are similar, but different, to the criteria of Theorem \ref{alpha}. Finally, for $L=-K_X-2T$, so that $\mu(X,L,T)=1$ and the second hypothesis of Theorem \ref{alpha} is vacuous, this is the algebro-geometric counterpart of a result due to Berman \cite[Theorem 4.5]{RB2}.\end{remark}

\begin{remark} As an application of Theorem \ref{alpha}, first take $L=-K_X$ and $T=\scO_X$. This is then the classical alpha invariant condition due to Tian, and there are several examples of varieties satisfying $\alpha(X,-K_X) > \frac{n}{n+1}$. Then, since the alpha invariant is a continuous function on the ample cone of $X$ \cite[Proposition 4.2]{RD}, one can perturb $T$ in any (ample) direction and the criteria of Theorem \ref{alpha} will still be satisfied. More explicit examples could also be achieved by using \cite[Theorem 4.4]{RD}.\end{remark}

Along the Aubin continuity method as in equation (\ref{twistedKE}), we recover the algebro-geometric analogue of an analytic result due to Sz\'ekelyhidi \cite[Section 3]{GS2}. Recall that we defined in equation $(\ref{twistedKEkstab})$ $$S(X) = \left\{\sup \beta: \left(X,-K_X,-\frac{1-\beta}{2}K_X\right) \textrm{ \ is \ uniformly \ twisted \ K-stable}\right\}.$$

\begin{corollary}\label{aubinkstab}  Along the Aubin continuity method we have $S(X) \geq \min\left\{\alpha(X)\frac{n+1}{n},1\right\}$. \end{corollary}

\begin{proof} In this case $L=-K_X$ and $T = -\frac{1-\beta}{2}K_X$, so $\mu(X,L,T)=\beta$. The second hypothesis of  Theorem \ref{alpha} is therefore vacuous, while the first condition states that $$\alpha(X,-K_X) \geq \beta\frac{n}{n+1}.$$ So Theorem \ref{alpha} shows that if $\alpha(X,-K_X) \geq \beta\frac{n}{n+1}$, then $\left(X,-K_X,-\frac{1-\beta}{2}K_X\right)$ is uniformly twisted K-stable. This proves the result. \end{proof}

\subsection{The twisted general type and Calabi-Yau cases}

In this section we give two sufficient geometric conditions for twisted K-stability, when either $K_X+2T$ is numerically trivial or ample.

\begin{theorem}\label{cy}Let $X$ be a $\Q$-Gorenstein variety with canonical divisor $K_X$. Suppose that $K_X+2T$ is numerically trivial, and let $L$ be an arbitrary ample line bundle.  \begin{itemize} \item[(i)] If $X$ is log canonical, then $X$ is twisted K-semistable. \item[(ii)] If $X$ is Kawamata log terminal, then $X$ is uniformly twisted K-stable.\end{itemize} \end{theorem}

\begin{remark} When $X$ is smooth, the analytic version of this result is essentially the Calabi-Yau theorem. \end{remark}

\begin{proof} We apply the formalism of  Proposition \ref{resofindet}. Since $K_X+2T$ is numerically trivial, the twisted slope $\mu(X,L,T)$ vanishes and the twisted Donaldson-Futaki invariant reduces to  \begin{equation} \DF(\scB,\scL-E,\scT) = (\scL-E)^n.(K_{\scB/X\times\pr^1}).\end{equation}If $X$ is log canonical, then this term is at least zero by log canonical inversion of adjunction \cite{MK} and Lemma \ref{inequalities} $(ii)$. When $X$ is Kawamata log terminal Lemma \ref{discreptouniform} $(iii)$ proves the result.

\end{proof}

We now turn to the case when $K_X+2T$ is ample, so that the twisted slope $\mu(X,L,T)$ is \emph{negative}. 

\begin{theorem}\label{gentypekstable}Let $X$ be a $\Q$-Gorenstein log canonical variety. Suppose that \begin{equation}\label{gentypeslope}-\mu(X,L,T) L\geq K_X+2T,\end{equation} in the sense that the difference is nef. Then $(X,L,T)$ is uniformly twisted K-stable with respect to the minimum norm. \end{theorem}

\begin{proof} We show that $\DF(\scB,\scL^r-E,\scT)>0$ for all flag ideals with $\scL^r-E$ relatively semi-ample over $\pr^1$. For notational simplicity, we assume $r=1$, the proof in the general case being the same. Since $X$ was assumed log canonical, the discrepancy term $(\scL-E)^n.K_{\scB/X\times\pr^1}$ is at least zero, by log canonical inversion of adjunction \cite{MK} and Lemma \ref{inequalities} $(ii)$. Using Proposition \ref{dfformulapf} this gives \begin{align*} DF(\scB,\scL-E,\scT) &\geq (\scL-E)^n.\left(\frac{n}{n+1}\mu(X,L,T)(\scL-E) + \scK_X+2\scT\right), \\  & =  (\scL-E)^n.( -\frac{1}{n+1}\mu(X,L,T)(\scL+nE) +  \\ & \ \ \ \ \ \ \  \ \ \ \  \ \ \ \ \ \ \ \ \ \  \ \ \ \  \ \ \ +  (\mu(X,L,T) L + \scK_X +2\scT)).\end{align*} Since $-\mu(X,L,T) L \geq K_X +2T$, Lemma \ref{inequalities} $(i)$ implies $$(\scL-E)^n.(\mu(X,L,T) L + \scK_X +2\scT))\geq 0.$$ Finally Remark \ref{explicitnorm} gives that $\DF(\scB,\scL-E,\scT)>\epsilon \|\scB\|_m$, as required. \end{proof}

\begin{remark} Theorem \ref{gentypekstable} is the algebraic counterpart of a result of Weinkove \cite{BW} and Song-Weinkove \cite[Theorem 1.2]{SW}, who proved that the Mabuchi functional is coercive in these classes under slightly weaker conditions, using the existence of critical points of the J-flow.

Note that in our setting we do not use the full positivity of the discrepancy term $K_{\scB/X\times\pr^1}$, one could strengthen Theorem \ref{gentypekstable} by using an alpha invariant condition similar to Theorem \ref{alpha}. On the analytic side, Li-Shi-Yao \cite[Theorem 1.1]{LSY} used this technique to prove coercivity of the Mabuchi functional using an alpha invariant condition on varieties of general type. As our topological condition, namely that $-\mu(X,L,0) L - K_X$ is ample, is a stronger condition on $L$ than needed by Song-Weinkove and Li-Shi-Yao, we do not provide the details. 

Again in the untwisted case, under similar conditions Ross-Thomas \cite[Theorem 8.5]{RT} and Panov-Ross \cite[Example 5.8]{PR} proved \emph{slope stability} of $(X,L)$, which by \cite[Example 7.8]{PR} is a strictly weaker concept than K-stability. \end{remark}

\subsection{Relationship between twisted and log K-stability}

We now remark on the differences between twisted K-stability and log K-stability. 

\begin{definition} Let $D\in |T|$ be a divisor, and let $(\scX,\scL)$ be a test configuration for $(X,L)$. Taking the closure of the orbit of $D$ under the $\C^*$-action $\scD$ gives a test configuration for $D$, therefore we can define the log Donaldson-Futaki invariant as $$\DF_{log}(\scX,\scL,\scD) = \frac{b_0a_1 - b_1a_0}{a_0} + \frac{\hat{b}_0 a_0 - b_0 \hat{a}_0}{a_0},$$ in the notation of Definition \ref{twisteddf}. We say that $(X,L,D)$ is \emph{log K-stable} if $$\DF(\scX,\scL,\scD)>0$$ for all test configurations $(\scX,\scL,\scD)$ such that $\|\scX\|_m >0$. \end{definition}

\begin{remark}There are two main differences between log K-stability and twisted K-stability. Most importantly, in the definition of twisted K-stability, we take a \emph{general} $D\in |T|$, whereas log K-stability fixes one specific $D$. Secondly, the definition of twisted K-stability requires a lift of the action on $\scX$ to some $\scT$, where $\scT$ restricts to $T$ on the non-zero fibres. There is no such lifting requirement in the definition of log K-stability, hence \emph{a priori} the set of test configurations used to check twisted K-stability is \emph{smaller} than for log K-stability. Of course when $L=cT$ for some $c\in \Q_{>0}$ the set of test configurations considered for each are the same. \end{remark}

We now show that log K-stability implies twisted K-stability. A similar result was proven by Sz\'ekelyhidi in the case $L=-K_X$ with $X$ smooth \cite[Theorem 6]{GS}. The proof proceeds by an examination of the relevant Donaldson-Futaki invariants. Note that our proof gives an \emph{explicit} form of the difference between the two Donaldson-Futaki invariants. In particular, there is strict inequality in the Donaldson-Futaki invariants if and only if the flag ideal $\scI$ associated to the test configuration has a component contained in $D$.

\begin{theorem}\label{logimpliestwisted} Suppose $(X,L,D)$ is log K-stable, with $D\in |2T|$. Then $(X,L,T)$ is twisted K-stable. \end{theorem}

\begin{proof} By Odaka-Sun \cite[Corollary 3.6]{OS2}, to check log K-stability, it is equivalent to show that $\DF_{log}(\scB,\scL^r-E,\scD)>0$ for all flag ideals with $\scL^r-E$ relatively semi-ample over $\pr^1$. By Proposition \ref{resofindet}, one can also check twisted K-stability by blowing-up flag ideals with $\scL^r-E$ relatively semi-ample over $\pr^1$ and $\scB$ normal. We show $$\DF_{log}(\scB,\scL^r-E,\scD) \leq \DF_{tw}(\scB,\scL^r-E,\scT),$$ where for clarity we add subscripts to the Donaldson-Futaki invariants to keep track of which is the twisted and log versions. Clearly this will show that log K-stability implies twisted K-stability.

From  \cite[Theorem 3.7]{OS2}, the relevant formula for $D \in |2T|$ is \begin{align*} \DF_{log}(\scB,\scL^r-E,\scD) = \frac{n}{n+1}&\mu(X,L,T)(\scL^r-E)^{n+1} + \\ & + (\scL^r-E)^n.(\scK_X + 2\scT +K_{\scB/(X,D)\times\pr^1,exc}).\end{align*} Here we have used notation as in the proof of Proposition \ref{dfformulapf}, which also states that $$\DF_{tw}(\scB,\scL^r-E,\scT) = \frac{n}{n+1}\mu(X,L,T)(\scL^r-E)^{n+1} +(\scL^r-E)^n.(\scK_X + 2\scT + K_{\scB/X\times\pr^1}).$$ The difference between these two expressions is $$(\scL^r-E)^n(2\scT - \pi^{-1}_*(D\times\pr^1)).$$ Since $(2\scT - \pi^{-1}_*(D\times\pr^1))$ is an effective exceptional divisor, Lemma \ref{inequalities} $(ii)$ implies $$\DF_{log}(\scB,\scL^r-E,\scD) \leq \DF_{tw}(\scB,\scL^r-E,\scT),$$ as required. Note in particular that $$2\scT - \pi^{-1}_*(D\times\pr^1)$$ is non-zero if and only if the flag $\scI$ has a component contained in $D$. That is, we have strict inequality if and only if the flag ideal associated to the test configuration has a component contained in $D$.  \end{proof}

\subsection{Singularities of twisted K-semistable varieties}\label{singularitiessection}

While the primary interested in K-stability is its relationship to the existence of cscK metrics, it is also an important concept from the point of view of forming moduli spaces. It is expected that one can form moduli spaces of K-stable objects with a compactification by K-semistable objects. Therefore it is important to understand the singularities which K-semistable objects can have. In the untwisted case this was done by Odaka \cite[Theorem 1.2]{O4}, we extend his result to the twisted case as follows.

% In the twisted setting, on the analytic side, when $X$ is smooth one should consider moduli of triples $(X,\omega,\alpha)$, where $\alpha \in c_1(T), \omega\in c_1(L)$ are smooth $(1,1)$-forms and there exists a solution to \begin{equation}\label{temptwisted}S(\omega) - \Lambda_{\omega}\alpha = C_{\alpha}.\end{equation} From the algebro-geometric point of view, twisted K-stability of $(X,L,T)$ is independed of the choice of $\alpha$. With this in mind, we make the following conjecture. 

%\begin{conjecture} If there exists a solution to equation (\ref{temptwisted}) for one $\alpha \in c_1(T)$, then there exists a solution for \emph{all} $\alpha\in c_1(T)$.\end{conjecture}

%Note that for a fixed $\alpha$,  the corresponding $\omega$ is \emph{unique}, if it exists. When $X$ is Fano, $L=-K_X$ and $T=-\beta K_X$, these are solutions to the equation \begin{equation}\Ric \omega = \beta \omega  + (1-\beta)\omega_0,\end{equation} with $\omega,\omega_0\in c_1(X)$. In this setting Sz\`ekelyhidi \cite{GS2} showed that the existence of solutions to this equation are \emph{independent} of $\omega_0$. This lends further credibility to the expectation that the analytic moduli of solutions to the equation $S(\omega) - \Lambda_{\omega}\alpha = C_{\alpha}$ should be equivalent to the algebraic moduli of triples $(X,L,T)$, in particular the analytic moduli should be independent of choice of $\alpha$.

%An important aspect of K-stability is its relation to singularities. 

\begin{theorem}\label{lcsings} Suppose $(X,L,T)$ is twisted K-semistable with $X$ normal and for an arbitrary choice of $L,T$. Then $X$ has log canonical singularities. \end{theorem}

We start with an arbitrary flag ideal $\scI$ with blow-up $\scB$, with notation as in Proposition \ref{resofindet}, with $\scL-E$ relatively semi-ample over $\pr^1$. We need the following preliminary definition.

\begin{definition}Denote the dimension of the support of $\scI$ as $s$. The $S$-\emph{coefficient} of $\scI$ is defined to be  $$S_{(X,L)}(\scI) = \scL^s.(-E)^{n-s}.K_{\scB/X\times\pr^1}.$$ By \cite[p. 656, equation $(1)$]{O4} we have $$S_{(X,L)} = \scL^s.(\scL^r-E)^{n-s}.K_{\scB/X\times\pr^1},$$ since $s$ is the dimension of the support of $\scI$. \end{definition}

\begin{proof} Let $\scI$ be an arbitrary flag ideal as above. Since $\scL-E$ is relatively semi-ample, $\scL^r-E$ is also relatively semi-ample for all $r \gg 0$. Note that twisted K-stability of $(X,L,T)$ is equivalent to twisted K-stability of $(X,L^r,T)$ for all $r>0$. Consider the sequence of twisted Donaldson-Futaki invariants of test configurations for $(X,L^r,T)$  \begin{equation*} \DF(\scB,\scL^r-E,\scT) = \frac{n}{n+1}\mu(X,L^r,T)(\scL^r-E)^{n+1} +(\scL^r-E)^n.(\scK_X + 2\scT + K_{\scB/X\times\pr^1}).\end{equation*} Note that the twisted slope is a rational function in $r$ of degree $-1$. We consider these Donaldson-Futaki invariants as a polynomial in $r$, with leading term $$\DF(\scB,\scL^r-E,\scT) = c S_{(X,L)}(\scI) r^{s} + O(r^{s-1}),$$ for some positive constant $c$. Denote $K_{\scB/X\times\pr^1} = \sum a_i E_i$, and suppose $a_i \leq 0$ for all $i$ with strict inequality for some $i$. Restricting to hyperplane sections of $\scL$ and applying Lemma \ref{inequalities} $(ii)$, as in \cite[Corollary 3.7]{O4}, we have $$\scL^s.(\scL^r-E)^{n-s}.K_{\scB/X\times\pr^1} < 0.$$ By definition of the $S$-coefficient of $\scI$ the Donaldson-Futaki invariant is then strictly negative, therefore $(X,L^r,T)$ is not twisted K-semistable. 

Suppose $X$ is not log canonical. Then \cite[656]{O4} constructs a flag ideal with the desired properties.  \end{proof}

It is natural to expect that Theorem \ref{lcsings} can be extended to the case when $X$ is not normal, under certain conditions, as in \cite{O4}.

When the variety is Fano and $L=-K_X$, Odaka strengthened the above result to the implication of Kawamata log terminal singularities \cite[Theorem 1.3]{O4}. We now apply his method to the twisted Fano setting. Note that the following result is new for general $L$ even when $T=\scO_X$. 

\begin{theorem}\label{fanoklt} Suppose $(X,L,T)$ is twisted K-semistable for some $L,T$, and suppose further that $\mu(X,L,T) L + K_X + 2T$ is nef. Then $X$ has Kawamata log terminal singularities. \end{theorem} 

\begin{proof} By Theorem \ref{lcsings} we can assume $X$ is log canonical. Suppose that $X$ is log canonical but not Kawamata log terminal. In this case \cite[659]{O4} constructs a flag ideal such that $K_{\scB/X\times\pr^1} = 0$. In particular, the Donaldson-Futaki invariant is given as \begin{align*}\DF(\scB,\scL-E,\scT) &= \frac{n}{n+1}\mu(X,L,T)(\scL-E)^{n+1} +(\scL-E)^n.(\scK_X + 2\scT), \\ &=(\scL-E)^{n}(-\frac{1}{n+1}\mu(X,L,T)(\scL+nE) +  \\ & \ \ \ \ \ \ \  \ \ \ \  \ \ \ \ \ \ \ \ \ \  \ \ \ \  \ \ \ \  \ \ \ \   +  (\mu(X,L,T) +\scK_X + 2\scT)).\end{align*} By Lemma \ref{inequalities} $(ii)$, since $\mu(X,L,T) L + K_X + 2T$ is nef, the second term is less than or equal to zero. Similarly by Lemma \ref{inequalities} $(i)$, the first term in the Donaldson-Futaki invariant is strictly negative. Therefore we have constructed a flag ideal with strictly negative Donaldson-Futaki invariant, a contradiction. \end{proof}

Combining the previous results with Theorem \ref{introgt} and Theorem \ref{introcy} gives the following.

\begin{corollary} Suppose $(X,L,T)$ is a normal variety. Suppose that \begin{itemize}\item[(i)] The twisted slope satisfies $$-\mu(X,L,T) L\geq K_X+2T.$$ Then $X$ is uniformly twisted K-stable with respect to the minimum norm if and only if $X$ is log canonical.\item[(ii)] The twisted slope is zero, i.e. $(X,L,T)$ is numerically twisted Calabi-Yau. Then $X$ is twisted K-semistable if and only if $X$ is log canonical. \end{itemize}\end{corollary}

\section{Norms on test configurations}

The goal of this section is to clarify the various notions of triviality of test configurations. For reference we recall the definitions of each of the three ways of measuring triviality of test configurations.

\begin{definition}\label{almosttrivial} We say that a test configuration $(\scX,\scL)$ is \emph{almost trivial} if it is $\C^*$-equivariantly isomorphic to the trivial test configuration $X\times\C$ outside of a closed subscheme of codimension two. This is equivalent to $(\scX,\scL)$ having normalisation $X\times\C$ with trivial $\C^*$-action on $X$ \cite{JS3}.\end{definition}

\begin{remark} The notion of almost trivial test configurations was introduced Stoppa \cite{JS3}, to avoid a pathology related to test configurations with embedded points noted by Li-Xu \cite[Section 8.2]{LX}.\end{remark}

\begin{definition} Let $(\scX,\scL)$ be a test configuration with infinitesimal generator of the $\C^*$-action on the central fibre $A_k$. The trace of the square of the weights of the $\C^*$-action on $H^0(\scX_0, \scL_0^k)$ is a polynomial of degree $n+2$ for $k\gg 0$. Denoting this polynomial by $\tr(A_k^2) = d_0k^{n+2} + O(k^{n+1})$, we define the $L^2$\emph{-norm} of a test configuration to be $$\|\scX\|_2 = d_0 - \frac{b_0^2}{a_0}.$$ \end{definition}

\begin{remark} This norm on test configurations was introduced by Donaldson \cite{D2}, in order to normalise the Donaldson-Futaki invariant to obtain a lower bound on the Calabi functional, as we did in Proposition \ref{calabifunctional}. The $L^2$-norm is also expected to have direct analytic relevance, see \cite[Section 3.1.1]{GS5}. \end{remark}

\begin{definition} Let $(\scX,\scL)$ be a test configuration. Assume the central fibre splits into irreducible components $\scX_{0,j}$, which by flatness must have dimension $n$. The $\C^*$-action fixes each component, hence we have a $\C^*$-action on $H^0(\scX_{0,j}, \scL_{0,j})$. Denote by $a_{0,j},b_{0,j}$ the leading terms of the corresponding polynomials. Let $\lambda_j$ be the minimum weight of the $\C^*$-action on the reduced support of the central fibre $X_0$, so that geometrically $\lambda$ is the weight of the $\C^*$-action on a general section of $|\scL_{0,j}|$. We define the \emph{minimum norm} of a test configuration to be $$\|\scX\|_m = \sum_j (b_{0,j} - \lambda_j a_{0,j}).$$ 

Let $D\in |L|$ be a divisor, with corresponding Hilbert and weight polynomials \begin{align*}\tilde{h}(k)&=\dim H^0(D,L|_D^k) = \tilde{a}_{0,D}k^{n-1}+O(k^{n-2}), \\ \tilde{w}(k)&=\wt H^0(\scD_0,\scL_0|_{\scD_0}^k)= \tilde{b}_{0,D}k^{n}+O(k^{n-1}). \end{align*} Since the term $\tilde{b}_{0,D}$ is constant outside a Zariski closed subset of $|D|$ \cite[Lemma 9]{LS}, we define $\tilde{b}_0$ to equal this general value. By Remark \ref{equivalentminnorm}, the minimum norm is also given as \begin{equation}\label{niceminnorm}\|\scX\|_m =  \frac{\tilde{b}_0a_0 - b_0\tilde{a}_0}{a_0}.\end{equation}\end{definition}

\begin{remark} This norm naturally occurs from the point of view of twisted cscK equations. However, since it is also linked to the coercivity of the Mabuchi functional, it is natural to expect that it is also important in the study of genuine cscK metrics. Indeed we provided in Section 3 several geometric situations in which one can show that a variety $(X,L)$ is uniformly K-stable with respect to the minimum norm.\end{remark}

\begin{theorem}\label{trivialitytheorem} Let $(\scX,\scL)$ be a test configuration. The following are equivalent. \begin{itemize}\item[(i)] $(\scX,\scL)$ is almost trivial. \item[(ii)] The $L^2$-norm $\|\scX\|_2$ is zero. \item[(iii)] The minimum norm $\|\scX\|_m$ is zero. \end{itemize}\end{theorem}

\begin{proof} $(i) \Rightarrow (iii)$ We wish to show that if the normalisation of $(\scX,\scL)$ is $X\times\C$, then the minimum norm of $(\scX,\scL)$ is zero. In fact we show more generally that the minimum norm is preserved under normalisation. We use equation (\ref{niceminnorm}). By \cite[Proposition 5.1]{RT}, the term $b_0$ is preserved by normalisation. Again using \cite[Proposition 5.1]{RT}, for \emph{any} fixed $D\in |L|$, the quantity $\tilde{b}^D_0$ is preserved by normalisation. Therefore the value $\tilde{b}_0$ is preserved. It follows that the minimum norm is preserved under normalisation.

$(ii) \Leftrightarrow (iii)$  The implication $(iii) \Rightarrow (ii)$ is contained in \cite[Theorem 12]{LS}, we note that their proof actually shows equivalence. 

Firstly, recall that by flatness the central fibre of any test configuration is equidimensional, that is, each irreducible component has the same dimension.

Note that the weight of a general section of $\scL^k_{0,j}$ is $k$ times the weight on a general section of $\scL_0,j$. In terms of integrals under an embedding of the test configuration, the minimum norm of a test configuration is given as \begin{equation}\label{highestweight}\|\scX\|_mk^{n+1} = \sum_j (b_{0,j} - \lambda_j a_{0,j})k^{n+1} = \sum_j \left(-\int_{|X_{0,j}|} h_A \frac{\omega_{FS}^n}{n!} - k\lambda_j \int_{|X_0,j|}\frac{\omega_{FS}^n}{n!}\right),\end{equation} where $X_{0,j}$ is the $j^{\mathrm{th}}$ irreducible component of the central fibre. Recall that the Hamiltonian function $h_A$ is given as $$h_A = -\frac{A_{jk}z_j\bar{z}_k}{|z|^2}.$$ So we see that the minimum norm is non-zero if and only if the action on at least one component of the reduced central fibre of the test configuration is non-trivial (this is perhaps clearest if one diagonalises the matrix $A$).

On the other hand, by Lemma \ref{dgtoag} $(vi)$, the $L^2$-norm of a test configuration is given as $$\|\scX\|_2 k^{n+2} = \int_{|X_0|} (h_A - \bar{h}_A)^2 \frac{\omega^n_{FS}}{n!}.$$ Similarly one sees that the $L^2$-norm of a test configuration is non-zero if and only if the action on at least one component of the reduced central fibre $|X_0|$ of the test configuration is non-trivial.

$(ii) \Rightarrow (i)$ We take $(\scX,\scL)$ to be a normal test configuration with zero minimum norm. We aim to show $(\scX,\scL)$ is isomorphic to $X\times\C$. Since $\scX$ is normal, there is flag ideal $\scI$ with blow-up $\scB$ with the same minimum norm. Moreover, on $\scB$ the minimum norm is given by Remark \ref{explicitnorm} as \begin{align*}\|\scX\|_m &=(\scL-E)^n\left(\frac{1}{n+1}(\scL+nE)\right), \\ &\geq 0,\end{align*} with equality if and only if $\scX$ is $\C^*$-equivariantly isomorphic to $X\times\C$ with trivial action on $X$, i.e. $\scI = (t^N)$ for some $N$. \end{proof}

%\begin{remark} It would be interesting to know if the $L^2$-norm and the minimum norm are Lipshitz equivalent, that is, if there exist constants $\epsilon_1,\epsilon_2>0$ such that for all test configurations we have $$\epsilon_1\|\scX\|_2 \leq\|\scX\|_m  \leq \epsilon_2\|\scX\|_2.$$ By Theorem \ref{stabilityofcscK} this would then imply that the existence of a twisted cscK metric implies uniform twisted K-stability with respect to the $L^2$-norm.\end{remark}

\printbibliography

\vspace{4mm}

\noindent Ruadha\'i Dervan, University of Cambridge,  UK. \\ R.Dervan@dpmms.cam.ac.uk

\end{document}